\documentclass[abstract=true,10pt]{scrartcl}

\RequirePackage{amsthm,amsmath,amsfonts,amssymb}
\RequirePackage{graphicx}

\usepackage[left=2cm,right=2.5cm]{geometry}
\usepackage{authblk} %For Footnote Style Authorlist
\usepackage{epsfig}
\usepackage{color}
\usepackage{dsfont}
\usepackage{bm}
\usepackage{verbatim}
\usepackage{stmaryrd}
\usepackage{bbm}			
\usepackage{enumerate} 
\usepackage{nicefrac}
\usepackage{xfrac}
\usepackage{units} 
\usepackage{mathtools}
\usepackage{pdfcomment}
\usepackage{url}
\usepackage[labelfont={bf,sf},font={footnotesize},labelsep=space,format=plain]{caption}
\usepackage{acronym}
\usepackage{mathrsfs}
\usepackage{fancyhdr}
\usepackage[T1]{fontenc}
\usepackage{caption, booktabs}
\usepackage{textcomp}
\usepackage{multirow}
\usepackage{lipsum}
\usepackage{makecell}
\usepackage[section]{placeins}
\usepackage{hologo}
\usepackage{xcolor}
\usepackage{tikz}

\usetikzlibrary{positioning,arrows}
\usetikzlibrary{arrows.meta,decorations.markings}
\usetikzlibrary{matrix}
\usetikzlibrary{matrix,fit}
\tikzset{%
  mleftdelimiter/.style={inner ysep=0pt, inner xsep=1ex,left delimiter=\{,label={[label distance=3mm]left:#1}}
}

\theoremstyle{plain}

\newtheorem{theorem}{Theorem}[section]
\newtheorem{lemma}[theorem]{Lemma}
\newtheorem{corollary}[theorem]{Corollary}
\newtheorem{proposition}[theorem]{Proposition}

\theoremstyle{remark}

\newtheorem{remark}[theorem]{Remark}
\newtheorem{example}{Example}

\newcommand{\E}{\mathbb{E}}

\newcommand{\N}{\mathbb{N}}

\newcommand{\R}{\mathbb{R}}
\newcommand{\PP}{\mathbb{P}}

\newcommand{\XX}{\mathbf{X}}

\newcommand{\ZZ}{\bold{Z}}

\newcommand{\cM}{\mathcal{M}}
\newcommand{\cC}{\mathcal{C}}

\newcommand{\cR}{\mathcal{R}}
\newcommand{\cA}{\mathcal{A}}
\newcommand{\cB}{\mathcal{B}}

\newcommand{\cW}{\mathcal{W}}

\providecommand{\keywords}[1]{\textbf{Keywords } #1}
\newcommand{\eqd}{\stackrel{\mathrm{d}}=}

\newcommand{\1}{\mathds{1}}

\newcommand{\Ch}{\mathfrak{Cb}}

\newcommand{\de}{\mathrm{\,d}}

\newcommand{\supp}{\mathop{\mathrm{supp}}}

\DeclareMathOperator{\v@r}{V@R}

\DeclareMathOperator{\av@r}{AV@R}

\providecommand{\keywords}[1]
{
  \small	
  \textbf{\textit{Keywords---}} #1
}

\makeatletter
\renewcommand{\maketitle}{\bgroup\setlength{\parindent}{0pt}
\begin{center}
  \textbf{\@title}
\end{center}

\begin{flushleft}
	\@author
\end{flushleft}\egroup
}
\makeatother

\title{\begin{LARGE}Quantifying and estimating dependence via sensitivity of conditional distributions\end{LARGE}}

\date{}

\author[1,a]{Jonathan Ansari}
\author[1,2,b]{Patrick B. Langthaler}
\author[1,c]{Sebastian Fuchs}
\author[1,d]{Wolfgang Trutschnig}
\affil[1]{\begin{small} Department of Artificial Intelligence and Human Interfaces, University of Salzburg, Austria \end{small}}
\affil[2]{\begin{small} Department of Neurology, Christian Doppler Klinik, Paracelsus Medical University, Salzburg, Austria\end{small}\bigskip}

\affil[a]{\begin{small}\url{jonathan.ansari@plus.ac.at}\end{small}}
\affil[b]{\begin{small}\url{patrickbenjamin.langthaler@stud.plus.ac.at}\end{small}}
\affil[c]{\begin{small}\url{sebastian.fuchs@plus.ac.at}\end{small}}
\affil[d]{\begin{small}\url{wolfgang@trutschnig.net}\end{small}}

\begin{document}
\maketitle

\begin{abstract}
Recently established, directed dependence measures for pairs $(X,Y)$ of 
random variables build upon the natural idea of comparing the conditional distributions 
of $Y$ given $X=x$ with the marginal distribution of $Y$. They assign 
pairs $(X,Y)$ values in $[0,1]$, the value is $0$ if and only if $X,Y$ are independent, and it is $1$ exclusively for $Y$ being a function of $X$. 
Here we show that comparing randomly drawn conditional distributions
with each other instead or, equivalently, analyzing how sensitive the 
conditional distribution of $Y$ given $X=x$ is on $x$, opens the door to 
constructing novel families of 
dependence measures $\Lambda_\varphi$ induced by general convex functions $\varphi: \R \rightarrow \R$, containing, e.g., Chatterjee's coefficient of correlation as special case. 
After establishing additional useful properties of $\Lambda_\varphi$ we 
focus on continuous $(X,Y)$, translate $\Lambda_\varphi$ to 
the copula setting, consider the $L^p$-version and establish an estimator which is strongly consistent in full generality. A real data example and a simulation 
study illustrate the chosen approach and the
performance of the estimator. Complementing the afore-mentioned results, 
we show how a slight modification of the construction underlying $\Lambda_\varphi$ 
can be used to define new measures of explainability generalizing the 
fraction of explained variance. 
\end{abstract}


\section{Introduction}
In classical statistics comparing two groups/populations is usually done by 
considering specific para\-meters like the mean or the variance of the two 
distributions \(\PP^{Y|X=x_1}\) and \(\PP^{Y|X=x_2}\,\). In other words, one studies expressions of the form \(\Delta(r(x_1),r(x_2))\) or \(\Delta(v(x_1),v(x_2))\) for some non-negative functional/distance \(\Delta\,,\)  where \(r(x)=\E[Y|X=x]\) denotes the regression function (conditional expectation) and \(v(x) = \mathbb{V}(Y|X=x)\) the conditional variance
of $Y$ given $X=x$, respectively.
Moving from the two-groups setting (modelled by a binary random variable \(X\)) to the general case
of an arbitrary conditioning variable \(X\,,\) it seems natural to examine how much the conditional 
distribution of \(Y\) given \(X\) changes if \(X\) changes, i.e., how sensitive the conditional distribution $\PP^{Y|X=x}$ of \(Y\) given $X=x$ 
is on \(x\). Intuitively, knowing the extent of sensitivity should provide information on how much information on 
\(Y\) we gain by knowing \(X\). Taking into account the distribution of the conditioning variable \(X\), 
one natural approach for quantifying sensitivity is to consider the average/expected value of \(\Delta(\PP^{Y\vert X=x_1},\PP^{Y\vert X=x_2})\) for a sample \(x_1,x_2\) of \(X\). Writing 
\(\mathbb{P}^X \otimes \mathbb{P}^X\) for the product measure of \(\mathbb{P}^X\) with itself one could therefore consider functionals of the form
\begin{equation}\label{eq:Lamda.allg}
  \Lambda_\Delta(Y|X) 
  := \alpha_\Delta^{-1} \, \int_{\mathbb{R}^2} \Delta (\PP^{Y\vert X=x_1},\PP^{Y\vert X=x_2}) 
  \de (\mathbb{P}^X \otimes \mathbb{P}^X)(x_1,x_2),
\end{equation} 
for some normalising constant \(\alpha_\Delta > 0\,.\)
As a matter of fact, various well-known measures can be represented in the form \eqref{eq:Lamda.allg}, including the portion of explained variance or Chatterjee's famous coefficient of rank correlation, see Section \ref{sec_setting_examples}. Comparing conditional distributions is 
certainly not a new idea, see, e.g., 
\cite{Dwork-2006,Evfimievski-2003,Rohde-2020} in the context of differential privacy. 
However, to the best of our knowledge, the (recently fast growing) lite\-rature on dependence measures  (also known as measures of predictability, i.e., measures attaining values in [0,1] and being minimal/maximal exclusively for the case of independence/perfect dependence) mainly focuses on comparing conditional and unconditional distributions, 
so on expressions of the form
\begin{align}\label{eq:cond-uncond}
    \int_{\mathbb{R}} \Delta (\PP^{Y\vert X=x},\PP^Y)
  \de \mathbb{P}^X(x)\,,
\end{align}
where \(\Delta\) refers either to a \(L^2\)-distance (see, e.g., \cite{Ansari-Fuchs-2022, chatterjee2021, chatterjee2020, sfx2022phi, gamboa2020, emura2021, sungur2005, siburg2013}),
the \(L^1\)-distance (see, e.g., \cite{fgwt2021,JGT,wt2011}),
some optimal transport cost function (see, e.g., \cite{Munk-2023}),
the maximum mean discrepancy (see, e.g., \cite{deb2020b}), or the Wasserstein distance (see, e.g., \cite{wiesel2022}).
\\
On the one hand, for \(\Delta\) denoting the 
\(L^2\)-distance between univariate distribution functions the quantities in \eqref{eq:Lamda.allg} and \eqref{eq:cond-uncond} can be shown to coincide, see Example \ref{Ex:Chatterjee} in the next section. On the other hand, in general these expressions may 
differ, so considering functionals of the form \(\Lambda_\Delta\) leads to novel measures of association and, in particular, to new dependence measures, which are the main focus of this contribution. 

The remainder of this paper is organized as follows: Section \ref{sec_setting_examples} 
first shows that Chatterjee's dependence measure can be represented in the form  $\Lambda_\Delta$ and proposes a family of functionals $\Delta=\Delta_\varphi$ based on 
convex functions $\varphi: \R \rightarrow \R$. Before focusing on these measures 
for the rest of the paper, we show that using the same construction but replacing the 
conditional distributions by the regression functions leads to a large class of 
so-called explainability measures which can be seen as generalization of the well-known
fraction of explained variance. 
Section \ref{sec_mop} proves that the afore-mentioned functionals $\Lambda_{\Delta_\varphi}=\Lambda_\varphi$ are in fact dependence measures, studies
further (invariance and order) properties, then focuses on continuous $(X,Y)$, and translates the dependence measures to the copula setting. 
Based on these results, in Section \ref{sec_estimation} a copula-based checkerboard 
estimator is proposed and 
shown to be strongly consistent in full generality, i.e., without any smoothness restrictions on the underlying bivariate copulas/distributions. A simulation study 
illustrating the performance of the estimator complemented by a read data example in Section \ref{sec:sim} round off the paper. 

In what follows we always consider real-valued general random variables \(X, Y\) defined on a common probability space \((\Omega,\cA,\PP)\,\). We will refer to 
\(X\) as the exogenous and to \(Y\) as the endogenous variable and assume that
the latter is non-degenerated, i.e., \(\PP^Y\) is not a one-point distribution.
Furthermore we will write $X'\eqd X$ if, and only if $X'$ and $X$ have the same distribution, i.e., if $\PP^{X'}=\PP^X$ holds.

\section{General setting and motivating examples}\label{sec_setting_examples}
Let $\cM$ denote the class of probability measures on \(\R\,.\)
For studying the sensitivity of conditional distributions 
we consider functionals $\Delta: \cM \times \cM \rightarrow [0,\infty)$ that may satisfy several desirable and natural properties such as symmetry, i.e., \(\Delta(\mu,\nu)=\Delta(\nu,\mu)\) for all \(\mu,\nu\in \cM\,,\) or positive definiteness, 
i.e., \(\Delta(\mu,\nu) \geq 0\) for all \(\mu,\nu\in \cM\) with equality if and only 
if \(\mu=\nu\,.\) \\
Given a bivariate random vector \((X,Y)\) on \((\Omega,\cA,\PP)\,,\) the mapping $\Lambda_\Delta(Y|X)$ is defined according to \eqref{eq:Lamda.allg}, 
whereby we implicitly assume that the mapping $(x_1,x_2) \mapsto \Delta \big(\PP^{Y\vert X=x_1},\PP^{Y\vert X=x_2}\big)$ is integrable. Moreoever, for the rest of this paper the normalising 
constant \(\alpha_\Delta\) is defined by 
\begin{align}\label{defnormconst}
    \alpha_{\Delta}:= \sup\left\{ \int_{\R^2}\Delta(\PP^{Y'|X'=x_1},\PP^{Y'|X'=x_2}) \de (\PP^{X'}\otimes \PP^{X'})(x_1,x_2): \, (X',Y')\in \cR(Y)\right\}
\end{align}
where $\cR(Y)$ denotes the family of all random vectors $(X',Y')\colon \Omega \to \R^2$ with 
$Y'\eqd Y$. Simple expressions for $\alpha_{\Delta}$ will be derived later on. 
Whenever the supremum is positive and finite we obviously have that
\(\Lambda_{\Delta}\in [0,1]\). By definition, \(\Lambda_\Delta\)
assigns every bivariate random vector $(X,Y)$ a non-negative number
which depends only on the distribution of \((X,Y).\) 

In the sequel we study properties of the functional \(\Lambda_\Delta\) for several classes of \(\Delta\) and characterize in particular the maximal elements determining
the value of the constant \(\alpha_\Delta\,.\)
Notice that for \(\Delta\) being positive definite obviously \(\Lambda_\Delta\) characterizes independence since independence of \(X\) and \(Y\) means that
the conditional distributions \((\PP^{Y|X=x})_{x\in \R}\) do not depend on \(x\), 
which for such $\Delta$ is equivalent to \(\Lambda_\Delta(Y|X) = 0\).

\begin{remark}
    Considering the special case of \(\Delta\) being a metric on $\cM$, the functional
    \(\Lambda_\Delta(Y|X)\) boils down to the average $\Delta$-distance of 
    two randomly (according to $\PP^X$) drawn conditional distributions 
    \(\PP^{Y|X=x_1}, \PP^{Y|X=x_2}\).
    If, for example, \(\Delta\) is the Wasserstein distance \(\cW,\) then the functional in \eqref{eq:Lamda.allg} is given by
    \(\Lambda_\Delta(Y|X) = \alpha_\Delta^{-1} \int_{\R^2} \cW(\PP^{Y|X=x_1},\PP^{Y|X=x_2}) \de \PP^X\otimes \PP^X(x_1,x_2)\), see \cite{wiesel2022} for measures of association with Wasserstein distances of 
    the form \eqref{eq:cond-uncond}. For a discussion of the case where \(\Delta\) corresponds to \(L^p\)-metrics we refer to Section \ref{secLp}.
\end{remark}

The subsequent two examples illustrate the broadness of the $\Lambda_\Delta$-approach according to (\ref{eq:Lamda.allg}) in the context of quantifying (directed) dependence 
as well as explainability in terms of the sensitivity of conditional distributions.
They show that well-known statistical concepts can either 
be expressed in terms of (\ref{eq:Lamda.allg}) with adequately chosen $\Delta$.

\begin{example}[Chatterjee's coefficient of correlation]~\label{Ex:Chatterjee}\\
Consider \(\Delta\) given by 
\begin{align*} 
  \Delta\big(\PP^{Y\vert X=x_1},\PP^{Y\vert X=x_2}\big) 
  & := \int_\R \Big( \PP(Y \geq y \vert X=x_1) - \PP(Y \geq y \vert X=x_2) \Big)^2 
  \de \PP^Y(y)\,.
\end{align*}
Then \(\Lambda_\Delta\) defined by \eqref{eq:Lamda.allg} coincides with Chatterjee's famous coefficient of correlation (see \cite{chatterjee2020, siburg2013})
\begin{comment}
using Fubini's theorem and change of coordinates the normalising constant \(\alpha_\Delta\) according to \eqref{defnormconst} simplifies to %\textcolor{red}{bin generell kein Fan von sehr kurz gehaltenen Beweisen, auch weil mühsam für Reviewer und Leser -> an manchen Stellen Zwischenschritt einfügen} \textcolor{blue}{okay, würde ich hier eher so lassen, sonst müsste man \(Y_1,Y_2\) i.i.d. mit \(Y\) einfügen:} \textcolor{red}{-- wenn über \(\Omega\) integriert im Sinner der Konsistenz auch \(d P(\omega)\) schreiben (sonst steht ein Mal das Argument dabei und ein Mal nicht)}
\begin{align*}
  \alpha_\Delta 
  & = \left( \int_{\mathbb{R}^2} \int_\R \Big( \mathds{1}_{[y,\infty)}(y_1) - \mathds{1}_{[y,\infty)}(y_2) \Big)^2 
  \de \PP^Y(y) \de (\mathbb{P}^Y \otimes \mathbb{P}^Y)(y_1,y_2) \right)
  \\
  & = \left( \int_{\mathbb{R}} \int_\Omega \Big( \mathds{1}_{[y,\infty)}(Y_1(\omega)) - \PP(Y_1\geq y) - \mathds{1}_{[y,\infty)}(Y_2(\omega)) + \PP(Y_2\geq y) \Big)^2 \de \PP(\omega)
  \de \PP^Y(y)  \right)\\
  & = \left( \int_{\mathbb{R}} \int_\Omega \Big( \mathds{1}_{[y,\infty)}(Y_1(\omega)) - \PP(Y_1\geq y)\Big)^2 +\Big( \mathds{1}_{[y,\infty)}(Y_2(\omega)) - \PP(Y_2\geq y) \Big)^2 \de \PP(\omega)
  \de \PP^Y(y)  \right)\\
  & = \left( 2 \, \int_\R \mathbb{V}(\mathds{1}_{\{Y \geq y\}}) \de \PP^Y(y)\right)\,,
\end{align*}
where \(Y_1,Y_2,\) and \(Y\) are independent and identically distributed.
\end{comment}
%As a direct consequence, in this case \(\Lambda_\Delta\) defined by \eqref{eq:Lamda.allg} coincides with Chatterjee's famous coefficient of rank correlation 
since using change of coordinates we have 
\begin{eqnarray*}
  \Lambda_\Delta(Y \vert X) 
  & = & \alpha_\Delta^{-1} \, \int_{\mathbb{R}^2} \int_\R 
  \Big( \PP(Y \geq y \vert X=x_1) - \PP(Y \geq y \vert X=x_2) \Big)^2
  \de \PP^Y(y) 
  \de (\mathbb{P}^X \otimes \mathbb{P}^X)(x_1,x_2) 
  \\
  & = & 2 \, \alpha_\Delta^{-1} \, \int_{\R} \left( \int_\R \PP(Y \geq y \vert X=x)^2 \de \mathbb{P}^X(x) 
  - \left( \int_\R \PP(Y \geq y \vert X=x) \de \mathbb{P}^X(x) \right)^2 \right) \de \PP^Y(y)
  \\
  & = & \frac{\int_{\mathbb{R}} \mathbb{V}(\mathbb{P}(Y \geq y \vert X)) \,\de \mathbb{P}^Y(y)}{\int_\R \mathbb{V}(\mathds{1}_{\{Y \geq y\}}) \de \PP^Y(y)}.
\end{eqnarray*}
Thereby according to the afore-mentioned references again using change of 
coordinates and Fubini's theorem, and letting \(Y_1,Y_2,Y\) be i.i.d., the normalizing constant $\alpha_\Delta$
simplifies to (also see Theorem \ref{thelambdaphi}) 
\begin{align}\label{constchatt}
\begin{split}
    \alpha_\Delta &= \left( 2 \, \int_\R \mathbb{V}(\mathds{1}_{\{Y \geq y\}}) \de \PP^Y(y)\right)\\
    & = \left( \int_{\mathbb{R}} \int_\Omega \Big( \mathds{1}_{[y,\infty)}(Y_1(\omega)) - \PP(Y_1\geq y)\Big)^2 +\Big( \mathds{1}_{[y,\infty)}(Y_2(\omega)) - \PP(Y_2\geq y) \Big)^2 \de \PP(\omega)
  \de \PP^Y(y)  \right)\\
  & = \left( \int_{\mathbb{R}} \int_\Omega \Big( \mathds{1}_{[y,\infty)}(Y_1(\omega)) - \PP(Y_1\geq y) - \mathds{1}_{[y,\infty)}(Y_2(\omega)) + \PP(Y_2\geq y) \Big)^2 \de \PP(\omega)
  \de \PP^Y(y)  \right)\\
  & = \left( \int_{\mathbb{R}^2} \int_\R \Big( \mathds{1}_{[y,\infty)}(y_1) - \mathds{1}_{[y,\infty)}(y_2) \Big)^2 
  \de \PP^Y(y) \de (\mathbb{P}^Y \otimes \mathbb{P}^Y)(y_1,y_2) \right)\\
  &= \int_{\R^2} \Delta\left(\delta_{y_1},\delta_{y_2}\right) \de \PP^Y \otimes \PP^Y(y_1,y_2)\,,
\end{split}
\end{align}
 and where \(\delta_x\) denotes the Dirac measure at \(x\in \R\,.\) 
Hence, by the properties of Chatterjee's coefficient of correlation, the above 
choice of \(\Delta\) yields that the functional \(\Lambda_\Delta\) is a measure 
of dependence, i.e.,
\(\Lambda_\Delta(Y \vert X) \in [0,1]\), \(\Lambda_\Delta(Y \vert X) = 0\) if and only if \(X\) and \(Y\) are independent, and \(\Lambda_\Delta(Y \vert X) = 1\) exclusively 
if \(Y\) is completely dependent on \(X\,,\) i.e., there exists some measurable function \(f\colon \R\to \R\) with \(Y=f(X)\) almost surely.
\end{example}
Motivated by the above example, Section \ref{sec_mop} introduces and investigates a
large family of functionals \(\Lambda_\Delta\), where \(\Delta\) depends only on the difference of conditional distribution functions rescaled by some measurable 
function \(\varphi\colon \R\to \R\):  
\begin{align} \label{eq_special_delta}
  \Delta_\varphi\big(\PP^{Y\vert X=x_1},\PP^{Y\vert X=x_2}\big) 
  & := \int_{\R} \varphi\left(F_{Y|X=x_1}(y)-F_{Y|X=x_2}(y)\right) \de \PP^Y(y)
\end{align}
We will show in Theorem \ref{thelambdaphi} that for \(\varphi\) being convex and 
strictly convex in $0=\varphi(0)$ the functional $\Lambda_{\Delta_\varphi}=:\Lambda_\varphi$ 
is a dependence measures (with normalizing constant \(\alpha_{\Delta_\varphi}\) according to \eqref{constchatt}). 
\begin{example}[Cram{\'e}r-von-Mises indices; fraction of explained variance]~~\label{Ex:ExplVar}\\
%Suppose that \((X,Y)\) is a bivariate random vector with finite variance, i.e., $\mathbb{V}(Y) \in (0,\infty)$. 
Suppose that \(Y\) is square-integrable.
Denote by \(r(x):= \mathbb{E}(Y\vert X=x)\) the regression function
of \(Y\) given \(X=x\), and consider
\begin{eqnarray*}
  \Delta\big(\PP^{Y\vert X=x_1},\PP^{Y\vert X=x_2}\big)
  & := & \left( \mathbb{E} (Y\vert X=x_1) - \mathbb{E} (Y\vert X=x_2) \right)^2  =
  (r(x_1)-r(x_2))^2\,.
\end{eqnarray*}
%\begin{align}\label{exa1}
%\Delta(F_1,F_2):=\frac{1}{2\sigma^2} \left(\mathbb{E}(F_1) - \mathbb{E}(F_2)\right)^2\,,
%\end{align} 
%where $\mathbb{E}(F_i)$ denotes the expectation of a random variable $X_i$ with $X_i \sim F_i$, 
Then the normalizing constant \(\alpha_\Delta\) is given by
\begin{align*}
  \alpha_\Delta 
  & = \left( \int_{\mathbb{R}^2} \left( y_1 - y_2 \right)^2
  \de (\mathbb{P}^Y \otimes \mathbb{P}^Y)(y_1,y_2) \right)
    =  2 \, \mathbb{V}(Y)
\end{align*}
and the functional \(\Lambda_\Delta(Y|X)\) in \eqref{eq:Lamda.allg} coincides with the fraction of explained variance of $Y$ given $X$,
also known as Sobol index or Cram{\'e}r-von-Mises index (see \cite{gamboa2020}), i.e.,
\begin{eqnarray*}
  \Lambda_\Delta(Y|X) 
  & = & \alpha_\Delta^{-1} \, \int_{\mathbb{R}^2} \left( \mathbb{E} (Y\vert X=x_1) - \mathbb{E} (Y\vert X=x_2) \right)^2 
  \de (\mathbb{P}^X \otimes \mathbb{P}^X)(x_1,x_2) 
  \\
  %& = & \alpha_\Delta \, \int_{\mathbb{R}2} (r(x_1)-r(x_2))^2  \de(\mathbb{P}^X \otimes \mathbb{P}^X)(x_1,x_2) \\
  & = & 2 \, \alpha_\Delta^{-1} \, \left( \int_{\mathbb{R}} r^2(x) \de \mathbb{P}^X (x) 
  - \left( \int_{\mathbb{R}} r(x) \de \mathbb{P}^X (x) \right)^2 \right)
  \\
  & = & 2 \, \alpha_\Delta^{-1} \, \left( \mathbb{E}((r \circ X)^2)  - (\mathbb{E}(r \circ X))^2\right) 
  \\
  & = & 2 \, \alpha_\Delta^{-1} \, \mathbb{V}(r \circ X) 
    =   \frac{\mathbb{V}(\mathbb{E}(Y \vert X))}{\mathbb{V}(Y)}
    =   \frac{\int_{\mathbb{R}} \left( \mathbb{E} (Y\vert X=x) - \mathbb{E} (Y) \right)^2  
              \de \mathbb{P}^X (x)}{\mathbb{V}(Y)} \,.
\end{eqnarray*}
%i.e., $\Lambda_\Delta(Y|X)$ coincides with the fraction of explained variance of $Y$ given $X$
%also known as Sobol index or Cram{\'e}r-von-Mises index; 
Recall that in this case 
\(\Lambda_\Delta(Y \vert X) \in [0,1]\) and we have  
\(\Lambda_\Delta(Y \vert X) = 0\) if, and only if \(\mathbb{V}(r \circ X) = 0\) (which is the case if \(X\) and \(Y\) are independent but not vice versa), and 
\(\Lambda_\Delta(Y \vert X) = 1\) if, and only if \(Y\) is completely dependent on \(X\).
\\
In the case of an affine regression function $r(x)=ax + b$, 
the quantity $\Lambda_\Delta(Y|X)$ then coincides with the squared Pearson 
correlation $\rho$ of $X$ and $Y$, i.e.
$$\Lambda_\Delta(Y|X) = \frac{(Cov(X,Y))^2}{\mathbb{V}(X) \mathbb{V}(Y)}= \rho^2(X,Y)\,.$$ 
\end{example}
As demonstrated in Example \ref{Ex:ExplVar},
$\Lambda_\Delta(Y|X)$ coincides with the $L^2$-Cram{\'e}r-von-Mises index,
i.e., it coincides with the expected squared distance between the conditional 
and the unconditional expectation. 
This motivates to investigate another subclass of functionals $\Lambda_\Delta$ 
where \(\Delta\) now depends only on the difference of conditional expectations, weighted
again by some measurable function \(\psi\colon \R \to \R\), i.e., 
\begin{align}\label{eq_special_delta2}
  \Delta\big(\PP^{Y\vert X=x_1},\PP^{Y\vert X=x_2}\big)
  & := \psi \left( \mathbb{E} (Y\vert X=x_1) - \mathbb{E} (Y\vert X=x_2) \right).
\end{align}
This leads to a novel class of Cramer-von-Mises indices solely based on 
sensitivity of the conditional expectations.
For some specific choices of \(\psi\) in \eqref{eq_special_delta2}, the associated functional \(\Lambda_\Delta\) in \eqref{eq:Lamda.allg} relates the variability of the conditional distribution to the variability of the unconditional distribution as in \eqref{eq:cond-uncond}, see, e.g., \cite{Furman-2017} for several Gini-type measures of risk and variability.\\
For the special case where \(\Delta=\Delta_\psi\) is of the form \eqref{eq_special_delta2}, the functional \(\Lambda_{\Delta_\psi}=:\Delta_\psi\) 
is given by 
\begin{align}\label{eq:Lamda.delta2.phi}
    \Lambda_\psi(Y|X) 
    :&= \alpha_\psi^{-1} \int_{\R^2} \psi \left( \mathbb{E} (Y\vert X=x_1) - \mathbb{E} (Y\vert X=x_2) \right) \de (\PP^{X}\otimes \PP^{X})(x_1,x_2)
\end{align}
%\textcolor{blue}{\(\Lambda_\psi\) ist in Kapitel 3 bereits anderweitig definiert. Hier dann \(\Lambda_\psi\)? Ist keine mathematisch saubere Lösung, aber evtl. ein Kompromiss...}
with normalizing constant \(\alpha_\psi\) fulfilling
\[\alpha_\psi= \sup\left\{ \int_{\R^2} \psi \left( \mathbb{E} (Y'\vert X'=x_1) - \mathbb{E} (Y'\vert X'=x_2) \right) \de (\PP^{X'}\otimes \PP^{X'})(x_1,x_2): (X',Y')\in \cR(Y)\right\}\,,\]
whenever the supremum is finite and positive.

In order to assure that $\Lambda_\psi$ according to \eqref{eq:Lamda.delta2.phi} is a measure of explainability the following property will be key: 
A convex function \(\phi\colon \R\to \R\) 
is said to be strictly convex in \(0\) if \((\phi(-\varepsilon)+\phi(\varepsilon))/2 > \phi(0)\) holds for every \(\varepsilon>0\,.\)
Notice that in the subsequent theorem $\psi$ is also allowed to attain negative values.
\begin{theorem}[Measures of explainability]\label{thelambdapsi}~\\\label{thelambdapsi}
Suppose that \(\psi(2Y)\) and \(\psi(-2Y)\) are integrable and that 
\(\psi\colon \R\to \R\) is convex and strictly convex in \(0\) with \(\psi(0)=0\).
Then the constant \(\alpha_\psi\) is given by
\begin{align}\label{def_alpha_varpsi}
        \alpha_\psi = \int_{\R^2} \psi\left(y_1 - y_2\right) \de (\PP^{Y} \otimes \PP^{Y})(y_1,y_2)\,.
    \end{align}
    %whenever the integral exists. 
    Furthermore, \(\Lambda_\psi(Y|X)\) defined according to \eqref{eq:Lamda.delta2.phi} is a measure of explainability, i.e., it fulfills the following three properties:
    \begin{enumerate}[(i)]
    \item \label{thelambdapsi1} \(\Lambda_\psi(Y|X)\in [0,1]\,.\)
    \item \label{thelambdapsi2} \(\Lambda_\psi(Y|X)=0\) if, and only if \(\mathbb{V}(\mathbb{E}(Y \vert X)) = 0\), i.e., \(\mathbb{E}(Y \vert X) = \E[Y]\) 
    almost surely.
    \item \label{thelambdapsi3} \(\Lambda_\psi(Y|X)=1\) if, and only if 
    \(Y\) is completely dependent on \(X\), i.e., \(\mathbb{E}(Y \vert X) = Y\) 
    almost surely.
    \end{enumerate}
\end{theorem}
\begin{proof}
We first prove property \eqref{thelambdapsi1} and eq. \eqref{def_alpha_varpsi} and 
proceed as follows: 
Let \(Y_1\) and \(Y_2\) be independent copies of \(Y\). Then 
using convexity of \(\psi\) and the integrability assumption the right-hand side of eq. \eqref{def_alpha_varpsi} is finite since 
\begin{align*}
    \E\left|\psi(Y_1-Y_2)\right| &= \E\left|\psi\left(\frac{2 Y_1 + (-2Y_2)}{2}\right) \right| \\
    &\leq \E\left| \psi(2Y_1)+\psi(-2Y_2)\right| 
    \leq \E|\psi(2Y)| + \E|\psi(-2Y)| < \infty.
\end{align*}
If \(Y = f(X)\) almost surely, then using change of coordinates it yields 
\begin{align}\label{eqsupalphaphi}
\begin{split}
    & \int_{\R^2} \psi \left( \mathbb{E} (Y\vert X=x_1) - \mathbb{E} (Y\vert X=x_2) \right) \de (\PP^{X}\otimes \PP^{X})(x_1,x_2)\\
    &= \int_{\R^2} \psi \left( f(x_1) - f(x_2) \right) \de (\PP^{X}\otimes \PP^{X})(x_1,x_2)\\
    &= \int_{\R^2} \psi \left( y_1 - y_2 \right) \de (\PP^{f(X)}\otimes \PP^{f(X)})(y_1,y_2)\\
    &= \int_{\R^2} \psi\left(y_1 - y_2\right) \de (\PP^{Y} \otimes \PP^{Y})(y_1,y_2).
    \end{split}
\end{align}
Hence, using disintegration and Jensen's inequality it follows for all 
$(X',Y')\in \cR(Y)$ that
\begin{align}
    \nonumber 0 
    &= \psi(0) = \psi\left(\E(Y') - \E(Y')\right)
    \\
   \nonumber  &= \psi\left( \int_{\R^2} \E[Y'|X'=x_1] - \E[Y'|X'=x_2] \de (\PP^{X'} \otimes \PP^{X'})(x_1,x_2)\right)\\
   \label{lins1} &\leq \int_{\R^2} \psi\left( \E[Y'|X'=x_1] - \E[Y'|X'=x_2] \right) \de (\PP^{X'} \otimes \PP^{X'})(x_1,x_2)\\
   \nonumber  &= \int_{\R^2} \psi\left( \int_{\R^2} y_1 - y_2 \de (\PP^{Y'|X'=x_1} \otimes \PP^{Y'|X'=x_2})(y_1,y_2) \right) \de (\PP^{X'} \otimes \PP^{X'})(x_1,x_2)\\
   \nonumber  &\leq \int_{\R^2} \int_{\R^2} \psi\left(y_1 - y_2\right) \de (\PP^{Y'|X'=x_1} \otimes \PP^{Y'|X'=x_2})(y_1,y_2) \de (\PP^{X'} \otimes \PP^{X'})(x_1,x_2)\\
   \label{lins2} &= \int_{\R^2} \psi\left(y_1 - y_2\right) \de (\PP^{Y'} \otimes \PP^{Y'})(y_1,y_2)  \\
   \nonumber &\leq \sup\left\{ \int_{\R^2} \psi \left( \mathbb{E} (Y'\vert X'=x_1) - \mathbb{E} (Y'\vert X'=x_2) \right) \de (\PP^{X'}\otimes \PP^{X'})(x_1,x_2): (X',Y')\in \cR(Y)\right\}\\
   \nonumber  &= \alpha_\psi.
\end{align}
Using eq. \eqref{eqsupalphaphi} and applying eq. \eqref{lins1} to eq. \eqref{lins2} 
on the elements the supremum is taken over shows that the third inequality becomes an equality.
This proves eq. \eqref{def_alpha_varpsi} and implies \(\Lambda_\psi(Y|X)\in [0,1]\). \\
\eqref{thelambdapsi2}: Obviously \(\mathbb{V}(\mathbb{E}(Y \vert X)) = 0\) if, and only if \(\mathbb{E}(Y \vert X=x_1)=\E(Y)=\mathbb{E}(Y \vert X=x_2)\) for \(\PP^X \otimes \PP^X\)-almost all \(x_1,x_2 \in \R\). 
The latter, however, is equivalent to the integrand in \eqref{eq:Lamda.delta2.phi} being zero almost everywhere since \(\psi\) is strictly convex in \(0\) and \(\psi(0)=0\).\\
\eqref{thelambdapsi3}: %\textcolor{blue}{
If \(Y\) is completely dependent on \(X\),
%Then there exists some function \(f\) such that \(Y=f(X)\) almost surely which gives \(\mathbb{E}(Y \vert X) = f(X) = Y\) %, hence \(\mathbb{V}(\mathbb{E}(Y \vert X)) = \mathbb{V}(Y)\), 
then \eqref{eqsupalphaphi} and \eqref{def_alpha_varpsi} imply
\begin{align*}
    \alpha_\psi \, \Lambda_\psi(Y|X) 
    &=  \int_{\R^2} \psi \left( \mathbb{E} (Y\vert X=x_1) - \mathbb{E} (Y\vert X=x_2) \right) \de (\PP^{X}\otimes \PP^{X})(x_1,x_2)\\
    %&= \int_{\R^2} \psi \left( f(x_1) - f(x_2) \right) \de (\PP^{X}\otimes \PP^{X})(x_1,x_2)\\
    &= \int_{\R^2} \psi \left( y_1 - y_2 \right) \de (\PP^{Y}\otimes \PP^{Y})(y_1,y_2) = \alpha_\psi.
\end{align*}
Since \(\PP^Y\) is non-degenerate, complete dependence of \(Y\) on \(X\) yields \(\E(Y|X)=Y\) and we do not have  $Y= \E(Y)$ almost surely. Hence, using  \eqref{thelambdapsi1} and \eqref{thelambdapsi2} the constant \(\alpha_\psi\) is 
positive and \(\Lambda_\psi(Y|X)=1\) follows. \\
To show the reverse implication first recall that by Jensen's inequality 
and eq. \eqref{def_alpha_varpsi} we have
\begin{align}
\begin{split}\label{intgh}
    \alpha_\psi\Lambda_\psi(Y|X) &= \int_{\R^2} \psi\Big( \int_{\R^2} (y_1-y_2) \de \PP^{Y|X=x_1}\otimes \PP^{Y|X=x_2}(y_1,y_2)\Big) \de \PP^X\otimes \PP^X(x_1,x_2)\\
    &\leq \int_{\R^2} \int_{\R^2} \psi(y_1-y_2) \de \PP^{Y|X=x_1} \otimes \PP^{Y|X=x_2} \de \PP^X \otimes \PP^X(x_1,x_2) = \alpha_\psi.
    \end{split}
\end{align}
For \(\Lambda_\psi(Y|X)=1\) the above inequality becomes an equality there exists some Borel set \(G\subseteq \R^2\) with \(\PP^X\otimes \PP^X(G)=1\) such that for \((x_1,x_2)\in G\) and \(Y_i\sim \PP^{Y|X=x_i}\,,\) \(i\in \{1,2\}\,,\) \(Y_1\) and \(Y_2\) independent,
\begin{align}\label{eqJenseq}
    \text{either} \quad Y_1\leq Y_2 ~\, \text{almost surely} \quad \text{or} \quad Y_1\geq Y_2 ~\, \text{almost surely.}
\end{align}
In fact, otherwise Jensen's inequality would be strict due to strict convexity 
of $\psi$ at \(0=\psi(0)\), which contradicts \(\Lambda_\psi(Y|X)=1\). 
Letting \(G_x:=\{y\in \R\colon (x,y)\in G\}\) denote the $x$-cut of $G$, 
defining \(W:=\{x\in \R\colon \PP^X(G_x)=1\}\) and using disintegration obviously yields
\(\PP^X(W)=1\). To simplify notation let \(l(x)\) and \(u(x)\) denote the infimum and
the supremum of the support of \(\PP^{Y|X=x}\,,\) i.e.,
\begin{align}
    l(x):= \inf(\supp(\PP^{Y|X=x})) \leq \sup(\supp(\PP^{Y|X=x}))=:u(x).
\end{align}
Then according to \eqref{eqJenseq} for each \((x_1,x_2)\in G\) the open intervals fulfil
\((l(x_1),u(x_1))\cap (l(x_2),u(x_2)) = \emptyset\).
Finally define 
\begin{align}
    V:=\{x\in \R \colon \PP^{Y|X=x} \text{ is non-degenerate}\} = \{x\in \R\colon \mathbb{V}(Y|X=x)>0\}\,.
\end{align}
For completing the proof it suffices to show 
\begin{align}\label{eqts}
    \PP^X(V)=0\,
\end{align}
since it then follows that \(\PP^{Y|X=x}\) is degenerate for \(\PP^X\)-almost all \(x\in \R\,,\) implying that \(Y\) is completely dependent on \(X\).\\
Assume, on the contrary, that \(\PP^X(V)>0\) holds. Then considering \(\PP^X(W)=1\) it follows that 
\begin{align}\label{eqpvw}
    \PP^X(V\cap W)=\PP^X(V)>0\,.
\end{align}
Let \(x\in V\cap W\) be arbitrary but fixed. Then by contruction we have 
%Then, by definition of \(W\), it follows for \(\PP^X\)-almost all \(z\in \R\) and, thus, for \(\PP^X\)-almost all \(z\in V\cap W\) that
\begin{align}\label{eqstarx}
    (l(x),u(x))\cap (l(z),u(z)) = \emptyset
\end{align}
for all \(z\in \R\) with \((x,z)\in G\), so in particular for all \(z\in V\cap W \cap G_x\). Considering \(\PP^X(G_x)=1=\PP^X(W)\) it follows that \(0<\PP^X(V) = \PP^X(V\cap W \cap G_x)\). Hence, \eqref{eqstarx} holds in particular for
\(\PP^X\)-almost every \(z\in V\cap W\).\\
Finally, let \(q_1,q_2,q_3,\ldots\) be an enumeration of the rational numbers in $\R$.
Then defining $M_i$ by \(M_i:=\{x\in V\cap W\colon q_i\in (l(x),u(x))\}\) yields \(\bigcup_{i=1}^\infty M_i = V\cap W\) since the rationals are dense in $\R$. 
Using \eqref{eqpvw} as well as sub-\(\sigma\)-additivity yields
\begin{align}
    0 < \PP^X(V\cap W)= \PP^X\Big(\bigcup_{i=1}^\infty M_i\Big) \leq \sum_{i=1}^\infty \PP^X(M_i),
\end{align}
so there exists some \(i_0\in \N\) with \(\PP^X(M_{i_0})>0\).
The latter, however, contradicts \eqref{eqstarx}, so \(\PP^X(V)>0\) can not hold, and 
the proof is complete.
\end{proof}

\begin{remark}
    Our proof of assertion \eqref{thelambdapsi3} in Theorem \ref{thelambdapsi}
    requires the assumption of strict convexity of \(\psi\) in \(0\) and uses the fact
    that Jensen's inequality applied to the difference of two non-degenerate i.i.d. random variables is strict if \(\varphi\) is strictly convex in \(0\).
    Hence, equality can only hold if the underlying random variables are degenerate. 
    Due to the integral with respect to \(\PP^X\otimes \PP^X\,,\) it has to be shown that Jensen's inequality holds for the inner integral \(\PP^X\otimes \PP^X\)-almost surely. For \(\PP^X\) having at least one point mass the proof can be simplified 
    significantly, since then $\PP^X(V)=0$ follows immediately.  
\end{remark}

%%%%%%%%%%%%%%%%%%%%%%%%%%%%%%%%%%%%%%%%%%%%%%%%%%%%%%%%%%%%%%%%%%%%%%%%%%%%%%%%%%%%%%%
%%%%%%%%%%%%%%%%%%%%%%%%%%%%%%%%%%%%%%%%%%%%%%%%%%%%%%%%%%%%%%%%%%%%%%%%%%%%%%%%%%%%%%%
%\clearpage

\section{A new class of dependence measures}\label{sec_mop}

For the rest of the paper we now focus on the functionals $\Lambda_\varphi$
and derive simple sufficient conditions for \(\varphi\) in \eqref{eq_special_delta} assuring that \(\Lambda_\Delta(Y|X)\) defined by \eqref{eq:Lamda.allg} is a 
dependence measure. 
%}
Following the recent literature we refer to $\kappa$ as a \emph{(directed) dependence measure} (frequently also referred to as measure of predictability) 
for bivariate random vectors $(X,Y)$ 
if it satisfies the following axioms, see, e.g., \cite{Ansari-Fuchs-2022,bickel2022,chatterjee2020,deb2020,fan2022A,strothmann2022,wt2011} and 
the references therein:
\begin{enumerate}[({A}1)]
\item \label{prop1} $\kappa(Y|X) \in [0,1]$.
\item \label{prop2} $\kappa(Y|X) = 0$ if, and only if $Y$ and $X$ are independent.
\item \label{prop3} $\kappa(Y|X) = 1$ if and only if $Y$ is completely dependent on $X$,
i.e., there exists some measurable transformation $f: \R \to \R$ such that $Y = f(X)$ holds almost surely.
\end{enumerate}
\begin{comment}
In addition to the above-mentioned three axioms, it is desirable that additional information improves the predictability of \(Y\,.\)
This yields the following two closely related properties of a measure of predictability $\kappa$ which prove to be of utmost importance for this paper (cf, e.g, \cite{sfx2022phi,GJT}): 
\begin{enumerate}[({P}1)]
\item \label{prop.IGI} \emph{Information gain inequality}: 
$\kappa(Y|\XX) \leq \kappa(Y|(\XX,\ZZ))$ for all $\XX$, $\ZZ$ and $Y$.
\item \label{prop.CI} \emph{Characterization of conditional independence}:
$\kappa(Y|\XX) = \kappa(Y|(\XX,\ZZ))$ if and only if $Y$ and $\ZZ$ are conditionally independent given $\XX$.
\end{enumerate}
\end{comment}
Due to the axioms (A\ref{prop2}) and (A\ref{prop3}), the values \(0\) and \(1\) of a measure of predictability  \(\kappa\) have a clear interpretation. 
The meaning of \(\kappa\,\) attaining a specific value in the interval \((0,1)\), however, is not that unequivocal - nevertheless, considering that the values in 
$[0,1]$ are naturally ordered, they offer a clear comparative interpretation:  
the higher the value of \(\kappa\) the higher the extent of dependence of \( Y\) 
on \(X\).
%}
This aspect motivates the study of dependence orderings \(\prec\) that are compatible with \(\kappa\) in the following sense:
\begin{enumerate}[({P}1)] \setcounter{enumi}{0}
\item \label{prop.MON} \emph{Monotonicity}: 
If \((X,Y) \prec (X',Y')\), then $\kappa(Y|X) \leq \kappa(Y'|X')$.
\end{enumerate}
%\textcolor{violet}{
It is further desirable - particularly in the context of estimation - to have an understanding of the measures' behavior in terms of convergence:
%}
\begin{enumerate}[({P}1)] \setcounter{enumi}{1}
\item \label{prop.CON} \emph{Continuity}: 
If \((X_n,Y_n)_{n \in \mathbb{N}}\) converges (in some sense) to \((X,Y)\),
then \linebreak \(\lim_{n \to \infty} \kappa(Y_n|X_n) = \kappa(Y|X)\).
\end{enumerate}
Moreover the following invariance properties seem desirable and will be studied in the 
sequel:
\begin{enumerate}[({P}1)] \setcounter{enumi}{2}
\item \label{prop.INVX} \emph{Invariance of \(X\)}: 
For bijective \(h\) we have \(\kappa(Y|X)=\kappa(Y|h(X))\).
\item \label{prop.INVY} \emph{Invariance of \(Y\)}: 
If \(g\) is strictly increasing, then \(\kappa(Y|X)=\kappa(g(Y)|X)\) holds.
\item \label{prop.lawINV} \emph{Version-Invariance}: 
If \((X,Y)\eqd (X',Y')\,,\) then \(\kappa(X,Y)=\kappa(X,Y)\,.\)
\end{enumerate}
Note that property (P\ref{prop.lawINV}) is also misleadingly referred to as law invariance in the context of risk measures. 
%Law-invariance (P\ref{prop.lawINV}), i.e., \((X,Y)\eqd (X',Y')\) implies \(Y=f(X,\varepsilon)\) \(P\)-almost surely if and only if there exists \(\varepsilon'\eqd \varepsilon\) such that \(Y'=f(X',\varepsilon')\) \(P\)-almost surely because
%\begin{align}
%    \E[g(Y')|X'=x] = \E[g(Y)|X=x] = \E[g(f(x,\varepsilon))|X=x]
%\end{align}
%for \(P\)-almost all \(\xx\in \R^p\,.\) \textcolor{blue}{(correct)?}

\begin{comment}
For \(r\in [1,\infty)\,,\) consider the functional \textcolor{blue}{man könnte das auch für verschiedene \(y\) definieren}
\begin{align}\label{lambda1}
    \Lambda_r(Y|X) := \frac{\int_{\R^2}\int_\R  \left|F_{Y|X=x_1}(y) - F_{Y|X=x_2}(y) \right|^r \de P^Y(y) \de (P^X\otimes P^{X})(x_1,x_2)}{\int_{\R^2}  \int_\R \left| \1_{\{y_1\leq y\}} -\1_{\{y_2\leq y\}} \right| \de P^Y(y) \de (P^Y\otimes P^{Y})(y_1,y_2)}
\end{align}
\end{comment}

%%%%%%%%%%%%%%%%%%%%%%%%%%%%%%%%%%%%%%%%%%%%%%%%%%%%%%%%%%%%%%%%%%%%%%%%%%%%%%%%%%%%%%%
%\clearpage

\subsection{The class \(\Lambda_\varphi\)}\label{subsec_Lambda_phi}
For the special case where \(\Delta\) is of the form \eqref{eq_special_delta} the functional \(\Lambda_\Delta \) %\Lambda_\varphi
is given by 
\begin{align}\label{eq:Lamda.delta.phi}
    \Lambda_\varphi(Y|X) :&= \alpha_\varphi^{-1} \int_{\R^2}\int_{\R} \varphi(F_{Y|X=x_1}(y)-F_{Y|X=x_2}(y)) \de \PP^Y(y) \de (\PP^{X}\otimes \PP^{X})(x_1,x_2)
\end{align}
with normalizing constant \(\alpha_\varphi\) according to \eqref{defnormconst} defined by
\begin{align}\label{eqnormcon_alphaphi}
     \alpha_\varphi &:= \sup\left\{\int_{\R^2}\int_{\R} \varphi(F_{Y'|X'=x_1}(y)-F_{Y'|X'=x_2}(y)) \de \PP^{Y'}(y) \de (\PP^{Y'}\otimes \PP^{Y'})(y_1,y_2): (X',Y')\in \cR(Y)\right\}\,,
\end{align}
whenever the supremum exists and is positive.

For deriving sufficient conditions on \(\varphi\) in order for \(\Lambda_\varphi\) in \eqref{eq:Lamda.delta.phi} to characterize independence and perfect directed dependence, the following simple lemma is key.

\begin{lemma}\label{lemconvord}
    Let \(U\) be a real-valued random variable on the probability space \((\Omega,\cA,\PP)\) and let \(\cB\) and \(\cC\) be \(\sigma\)-algebras with \(\cB\subseteq \cC\subseteq \cA\,.\) Then
    \(\E[U|\cB]\leq_{cx} \E[U|\cC]\,,\) where \(\leq_{cx}\) denotes the convex order.
\end{lemma}

\begin{proof}
    Let \(\varphi\) be a convex function. Then, whenever the following expectations exist, one has
    %\textcolor{red}{eine eckige Klammer durch eine geschwungene ersetzen -> erleichtert das     Lesen wenn man nicht dauern mit bedingten Erwartungen arbeitet} \textcolor{blue}{würde beim Erwartungswertoperator schon gerne nur eine Klammervariante haben}
    \begin{align}
        \E[\varphi(\E[U|\cB])] = \E[\varphi(\E[\E[U|\cC]|\cB])] \leq \E[\E[\varphi(\E[U|\cC])|\cB]] = \E[\varphi(\E[U|\cC])]
    \end{align}
    applying the tower property and Jensen's inequality for conditional expectation.
\end{proof}

%If additionally \(\varphi(0) = 0\,,\) then strict convexity simplifies to \(\varphi(\varepsilon)/2 > \varphi(\varepsilon/2)\,.\)
%For such functions \(\varphi\,,\) 
Having this we can now formulate the following main result of this section.

\begin{theorem}[Measures of directed dependence]\label{thelambdaphi}~\\
Assume that \(\varphi: [-1,1] \rightarrow \R\) is convex and strictly convex in \(0\) with \(\varphi(0)=0\) and that \(Y\) is non-degenerate. Then the normalizing constant 
\(\alpha_\varphi\) in \eqref{eqnormcon_alphaphi} is given by
\begin{align}\label{thelambda0}
    \alpha_\varphi = \int_{\R^2}\int_\R  \varphi\left(\1_{\{y_1\leq y\}} - \1_{\{y_2\leq y\}}\right) \de \PP^Y(y) \de (\PP^Y\otimes \PP^Y)(y_1,y_2)
\end{align}
and \(\Lambda_\varphi(Y|X)\) defined according to \eqref{eq:Lamda.delta.phi} is a 
dependence measure, i.e.,
    \begin{enumerate}[(i)]
    \item \label{thelambda11} \(\Lambda_\varphi(Y|X)\in [0,1]\,,\)
    \item \label{thelambda12} \(\Lambda_\varphi(Y|X)=0\) if, and only if \(X\) and \(Y\) are independent,
    \item \label{thelambda13}\(\Lambda_\varphi(Y|X)=1\) if, and only if \(Y\) is perfectly dependent on \(X\).
    \end{enumerate}
\end{theorem}

%\textcolor{blue}{Beweis in den Anhang?} \textcolor{red}{definitiv hier drinnen lassen weil schlüsselresultat}
\begin{proof}
    \eqref{thelambda11}: For every \(y\in \R\) and every \(c\in [0,1]\,,\)  
    applying Lemma \ref{lemconvord} directly yields
    \begin{align*}
          \int_\Omega \varphi\left( \E[\1_{\{Y\leq y\}}|X] - c \right) \de \PP \leq \int_\Omega \varphi\left( \E[\1_{\{Y\leq y\}}|(X,Y)] - c \right) \de \PP
          = \int_\Omega \varphi\left( \1_{\{Y\leq y\}} - c \right) \de \PP\,.
    \end{align*}
    Hence, using disintegration, Jensen's inequality and Fubini's theorem it follows
     that
    \begin{align*}
        0 &= \varphi(0) = \int_\R \PP(Y \leq y) - \PP(Y \leq y)\, \de \PP^Y(y)\\
        & = \varphi\left( \int_\R \int_{\R^2} (F_{Y|X=x_1}(y) - F_{Y|X=x_2}(y)) \de (P^X \otimes \PP^X)(x_1,x_2) \de \PP^Y(y)\right)\\
        &\leq \int_{\R^2}\int_\R  \varphi\left(F_{Y|X=x_1}(y) - F_{Y|X=x_2}(y) \right) \de \PP^X(x_1) \de (\PP^X\otimes \PP^Y)(x_2,y)\\
        &\leq \int_{\R^2}\int_\R  \varphi\left(\1_{\{y_1\leq y\}} - F_{Y|X=x_2}(y) \right) \de \PP^Y(y_1) \de (\PP^X\otimes \PP^Y)(x_2,y)\\
        &= \int_{\R^2}\int_\R  \varphi\left(\1_{\{y_1\leq y\}} - F_{Y|X=x_2}(y) \right) \de \PP^X(x_2) \de (\PP^Y\otimes \PP^Y)(y_1,y)\\
        &\leq \int_{\R^2}\int_\R  \varphi\left(\1_{\{y_1\leq y\}} - \1_{\{y_2\leq y\}}\right) \de \PP^Y(y_2) \de (\PP^Y\otimes \PP^Y)(y_1,y)\\
        & = \int_{\R^2}\int_\R  \varphi\left(\1_{\{y_1\leq y\}} - \1_{\{y_2\leq y\}}\right) \de \PP^Y(y) \de (\PP^Y\otimes \PP^Y)(y_1,y_2) \\
        &= \sup\left\{\int_{\R^2}\int_{\R} \varphi(F_{Y'|X'=x_1}(y)-F_{Y'|X'=x_2}(y)) \de \PP^{Y'}(y) \de (\PP^{Y'}\otimes \PP^{Y'})(y_1,y_2): (X',Y')\in \cR(Y)\right\}\\
        &= \alpha_\varphi\,,
    \end{align*}
    %\textcolor{blue}{
    where the supremum is attained whenever \(Y=f(X)\) almost surely 
    for some measurable $f$. This implies \eqref{thelambda0} and yields 
    \(\Lambda_\varphi(Y|X)\in [0,1]\,.\)
    %}
    \\
    \eqref{thelambda12}: \(X\) and \(Y\) are independent, if and only if \(F_{Y|X=x}(\cdot) = F_{Y}(\cdot)\) for \(\PP^X\)-almost all \(x\). The latter, however, is 
    %\textcolor{violet}{
    equivalent to the integrand in \eqref{eq:Lamda.delta.phi} being zero \(\PP^X\otimes \PP^X\)-almost everywhere since
    %}
    \(\varphi\) is strictly convex in \(0\) and \(\varphi(0)=0\,\) holds.\\
    (iii) Suppose that \(Y=f(X)\) almost surely for some measurable \(f: \R \rightarrow \R\). 
    Then obviously \(F_{Y|X=x}(y) = \1_{\{f(x)\leq y\}}\), so applying Fubini's theorem 
    and change of coordinates yields
    \begin{align}\label{eqnocoun}
        \alpha_\varphi\,\Lambda_\varphi(Y|X) &=\int_{\R}\int_{\R^2} \varphi(F_{Y|X=x_1}(y)-F_{Y|X=x_2}(y)) 
        \de (\PP^{X}\otimes \PP^{X})(x_1,x_2) \de  \PP^Y(y) \\
     \nonumber   &= \int_{\R}\int_{\R^2} \varphi(\1_{\{f(x_1)\leq y\}}-\1_{\{f(x_2)\leq y\}}) 
        \de (\PP^{X}\otimes \PP^{X})(x_1,x_2) \de  \PP^Y(y) \\
    \nonumber    &= \int_{\R}\int_{\R^2} \varphi(\1_{\{y_1\leq y\}}-\1_{\{y_2\leq y\}}) 
        \de (\PP^{f \circ X}\otimes \PP^{f \circ X})(y_1,y_2) \de  \PP^Y(y) \\
     \nonumber   &=\alpha_\varphi.
    \end{align}
    %\textcolor{blue}{
    Since \(Y\) is non-degenerate, \eqref{thelambda12} and \eqref{thelambda0} imply that \(\alpha_\varphi\) is positive and finite, so \(\Lambda_\varphi(Y|X)=1\) holds.
    %}
    \\
    %In other words: In case of perfect dependence of \(Y\) on \(X\) the functional 
    %\(\Lambda_\varphi(Y|X)\) is maximal.\\
    It remains to prove the reverse implication, i.e., that
    \(\Lambda_\varphi(Y|X)=1\) implies \(Y=f(X)\) almost surely for some $f$, which can be done as follows: First, we show that the condition \(\Lambda_\varphi(Y|X)=1\) implies the existence of some \(\PP^Y\)-null set \(N\) such that for fixed \(y\in N^c\) there exists a \(\PP^X\)-null set \(N_y\) with
     \begin{align}\label{proofmaxeltyp11}
        F_{Y|X=x}(y)\in \{0,1\} \quad \text{for all } x\in N_y^c\,.
    \end{align}
    We then prove that there exists some \(\PP^X\)-null set 
    \(N_*\) fulfilling
    \begin{align}\label{proofmaxeltyp12}
        F_{Y|X=x}(y) \in \{0,1\} \quad \text{for all } y\in \R \text{ and for all } x\in N_*^c\,.
    \end{align}
    Having this, proceeding like in the proofs of \cite[Theorem 9.2. (iii)]{chatterjee2021} or \cite[Lemma 12]{wt2011} yields complete dependence of \(Y\) on \(X\), which completes the proof.\\
    In order to show \eqref{proofmaxeltyp11},
    fix \(y \in \mathbb{R}\) and consider  
    \(a: \mathbb{R} \rightarrow \mathbb{R}\), defined
    by \(a(x)=\PP(Y\leq y\mid X=x)\). 
    %\textcolor{red}{die ganze erste lange Abschätzung die hier drinnen war ist mMn nicht notwendig, weil sie auch direkt aus der zweiten Abschätzung von \(\varphi\big(a(x_1)-a(x_2)\big)\) folgt (siehe unten) -> habe sie rausgenommen}
%    Then setting
%    \begin{align*}
%        c_1 &:=a(x_1)-a(x_1)a(x_2) \in [0,1] \\
%        c_{-1}&:=a(x_2)-a(x_1)a(x_2) \in [0,1] \\
%         c_0 &:=1-a(x_1)-a(x_2)+2a(x_1)a(x_2) \in [0,1] 
%    \end{align*}
%    we obviously have \(c_1 + c_{-1}+c_0 =1\). 
%    Hence using convexity of \(\varphi\) it follows that
%\begin{align*}
%    &\int_{\R^2} \varphi\left(\1_{\{y_1\leq y\}} - \1_{\{y_2\leq y\}}\right) \de (\PP^Y\otimes \PP^Y)(y_1,y_2)\\
%    & = \varphi(-1) \cdot\PP(Y\leq y) (1-\PP(Y\leq y)) + \varphi(1) \cdot\PP(Y\leq y) (1-\PP(Y\leq y))\\
%    & = \int_{\R^2} \varphi(-1) \big[\PP(Y\leq y\mid X=x_2)- \PP(Y\leq y\mid X=x_1) \PP(Y\leq y\mid X=x_2)\big]\\
%    & \quad\quad + \varphi(1) \big[\PP(Y\leq y\mid X=x_1)- \PP(Y\leq y\mid X=x_1) \PP(Y\leq y\mid X=x_2)\big] \de (\PP^X\otimes \PP^X) (x_1,x_2)\\
%    & = \int_{\R^2} \varphi(-1) \left(a(x_2)-a(x_1)a(x_2)\right) + \varphi(1) (a(x_1)-a(x_1)a(x_2)) + \underbrace{\varphi(0)}_{=0} \, c_0 \de (\PP^X\otimes \PP^X) (x_1,x_2) \\    
%    &\geq \int_{\R^2} \varphi\big(-1 \cdot (a(x_2)-a(x_1)a(x_2)) + 1 \cdot (a(x_1)-a(x_1)a(x_2) + 0 \cdot c_0 \big) \de \PP^X\otimes \de \PP^X(x_1,x_2) \\
 %   &=  \int_{\R^2} \varphi\big(a(x_1)-a(x_2)\big) \de \PP^X\otimes \de \PP^X(x_1,x_2). 
%\end{align*}
We first establish an upper bound for the mapping  
\((x_1,x_2) \mapsto \varphi\big(a(x_1)-a(x_2)\big)\)
and procced as follows: Setting
\begin{align*}
    c_1&:= a(x_1)-\min\{a(x_1),a(x_2)\} \in [0,1] \\
    c_{-1}&:= a(x_2)-\min\{a(x_1),a(x_2)\} \in [0,1] \\
    c_0&:=1-a(x_1)-a(x_2) + 2 \min\{a(x_1),a(x_2)\} \in [0,1]
\end{align*}
we obviously have \(c_{1} + c_{-1} + c_0=1\), so using convexity yields
\begin{align} \label{eqinconv1}
    &\varphi(a(x_1)-a(x_2)) \\
    \nonumber &= \varphi\big(1\cdot (a(x_1)-\min\{a(x_1),a(x_2)\}) - 
    1 \cdot (a(x_2)-\min\{a(x_1),a(x_2)\}) + 0 \cdot c_0 \big) \\
    \nonumber &\leq \varphi(1) (a(x_1)-\min\{a(x_1),a(x_2)\}) + \varphi(-1) (a(x_2)-\min\{a(x_1),a(x_2)\}) + \underbrace{\varphi(0)}_{=0} \cdot c_0 \\
    \nonumber &= \varphi(1) \, a(x_1) + \varphi(-1) \, a(x_2) - \underbrace{(\varphi(1)+\varphi(-1))}_{>0} \, \underbrace{\min\{a(x_1),a(x_2)\}}_{\geq a(x_1)a(x_2)}\\
    \nonumber &\leq \varphi(1) \, a(x_1) + \varphi(-1) \, a(x_2) - (\varphi(1)+\varphi(-1)) \, a(x_1)a(x_2)\\
    \nonumber &= \varphi(1) \, a(x_1)(1-a(x_2)) + \varphi(-1)\, a(x_2)(1-a(x_1))\\
    \nonumber &= \int_{\R^2} \varphi\left(\1_{\{y_1\leq y\}} - \1_{\{y_2\leq y\}}\right) \de (\PP^{Y|X=x_1}\otimes \PP^{Y|X=x_2})(y_1,y_2). 
\end{align}
Integrating both sides over \(\mathbb{R}^2\) with respect to \(\PP^X\otimes \PP^X\)
and using disintegration it follows that
\begin{align}\label{eq:upper}
& \int_{\R^2} \varphi(a(x_1)-a(x_2)) \de \PP^X\otimes \PP^X(x_1,x_2)  \\   
\nonumber & \leq \int_{\R^2}\int_{\R^2} \varphi\left(\1_{\{y_1\leq y\}} - \1_{\{y_2\leq y\}}\right) \de (\PP^{Y|X=x_1}\otimes \PP^{Y|X=x_2})(y_1,y_2) \, \de \PP^X\otimes \PP^X(x_1,x_2) \\
\nonumber &=\int_{\R} \int_\R \int_\R\int_{\R} \varphi\left(\1_{\{y_1\leq y\}} - \1_{\{y_2\leq y\}}\right) \de \PP^{Y|X=x_1}(y_1) \de \PP^X(x_1) \de \PP^{Y|X=x_2}(y_2) \de \PP^X(x_2) \\
\nonumber &= \int_{\R^2} \varphi(\1_{\{y_1\leq y\}} -\1_{\{y_2\leq y\}}) \de \PP^Y\otimes \PP^Y(y_1,y_2).
\end{align}
Consequently, if equality holds in \eqref{eq:upper}, then for 
\(\PP^X\otimes \PP^X\)-almost all \((x_1,x_2)\in \R^2\) inequality \eqref{eqinconv1}
becomes an equality. In particular, there exists some Borel set \(G \subseteq \R^2\) 
with \(\PP^X\otimes \PP^X (G)=1\) such that 
\begin{align*}
    \min\{a(x_1),a(x_2)\} = a(x_1) a(x_2)
\end{align*}
holds for all \((x_1,x_2)\in G\). 
Considering that the latter is equivalent to the fact that
\(a(x_1) \in \{0,1\}\) or \(a(x_2) \in \{0,1\}\) for all \((x_1,x_2)\in G\)
we conclude that \(a(x) \in \{0,1\}\) for \(\PP^X\)-almost every \(x \in \R\). % \textcolor{blue}{i.e., \eqref{proofmaxeltyp11} is shown.} \\
Summing up, for every fixed \(y \in \R\) we have shown that equality in (\ref{eq:upper})
implies that \(a(x)=\PP(Y\leq y\mid X=x) = F_{Y|X=x}(y) \in \{0,1\}\) for 
\(\PP^X\)-almost every \(x \in \R\). \\
Now, to prove \eqref{proofmaxeltyp11}, suppose that \(\Lambda_\varphi(Y|X)=1\), i.e., that we have
\begin{align}\label{proofmaxeltyp13}
\begin{split}
    &\int_\R \int_{\R^2} \varphi(a(x_1)-a(x_2)) \de \PP^X\otimes \PP^X(x_1,x_2) \de \PP^Y(y)\\
    &= \int_{\R} \int_{\R^2} \varphi(\1_{\{y_1\leq y\}} -\1_{\{y_2\leq y\}} \de \PP^Y\otimes \PP^Y(y_1,y_2) \de \PP^Y(y)\,.
    \end{split}
\end{align}
%\textcolor{blue}{Version 1 für \eqref{proofmaxeltyp12}: In order to prove \eqref{proofmaxeltyp12}, }
%suppose that \(\Lambda_\varphi(Y|X)=1\).
Then, due to (\ref{eq:upper}),
there exists some 
Borel set \(E \subseteq \R\) with \(\PP^Y(E)=1\) such that for every 
\(y \in E\) we have equality in \eqref{eq:upper}. Due to the previous step, 
for \(N:=E^c\) \eqref{proofmaxeltyp11} holds true for all \(y\in E\,,\) i.e., for every \(y \in E\) there exists a \(\PP^X\)-null set \(N_y\) with \(F_{Y|X=x}(y) \in \{0,1\}\) for all \(x\in N_y^c\).\\
In order to prove \eqref{proofmaxeltyp12} first note that by disintegration we also have
\(\PP^{Y|X=x}(E)=1\) for \(\PP^X\)-almost every \(x \in \R\,.\) Hence we may
proceed as follows:
Define the set \(N \in \mathcal{B}(\R^2)\) by 
\begin{align*}
    N:=\{(x,y) \in \R^2: F_{Y|X=x}(y) \in (0,1)\}
\end{align*}
and suppose that \(\PP^{(X,Y)}(N)>0\). Setting \(N_x := \{y \in \R: (x,y) \in N\}\)
disintegration yields 
\begin{align*}
    0 < \PP^{(X,Y)}(N) = \int_\R \PP^{Y|X=x}(N_x) \de \PP^X(x), 
\end{align*}
so the set \(O \subseteq \mathcal{\R}\), given by 
\(O=\{x \in \R: \PP^{Y|X=x}(N_x) >0 \textrm{ and } \PP^{Y|X=x}(E) =1\}\) fulfills
 \(\PP^{X}(O)>0 \). For every \(x \in O\) we therefore have 
 \(N_x \cap E \neq \emptyset\), which is absurd. Consequently,  \(\PP^{(X,Y)}(N)=0\), 
 whence \(\PP^{(X,Y)}(N^c)=1\) holds. Disintegrating once more yields a 
 \(\PP^X\)-null set \(N_*\) such that 
 for all \(x\in N_*^c\) the conditional distribution function 
 \(y \mapsto F_{Y|X=x}(y)\) only attains values in \(\{0,1\}\), which proves \eqref{proofmaxeltyp12}.
\end{proof}

\begin{remark}
\begin{enumerate}[(a)]
    \item If \(\varphi(x) = x^2\), then \(\Lambda_\varphi\) in \eqref{eq:Lamda.delta.phi} reduces to Chatterjee's coefficient of correlation, see Example \ref{Ex:Chatterjee}. However, in general, the comparison of conditional distributions via 
    \(\Lambda_\varphi\) does not boil down to comparing conditional and unconditional distribution, see Example \ref{Ex.FGM.Frechet}. In other words, Theorem \ref{thelambdaphi} provides a new class of dependence measures. 
    \item In Theorem \ref{thelambdaphi}, the assumption of strict convexity 
    of $\varphi$ in \(0\) is on the one hand necessary for the characterization of independence. On the other hand, it is also used in the proof that \(\Lambda_\varphi(Y|X)=1\) implies complete dependence. 
    Note that Theorem \ref{thelambdaphi} only requires that \(\varphi\) is 
    strictly convex in $0$ but not necessarily strictly convex on the full 
    interval \([-1,1]\). Hence, each of the functions \(x\mapsto \max\{x,0\}\,,\) \(x\mapsto \exp(c x)-1\) for \(c\ne 0\,,\) and \(x\mapsto |x|^p\,,\) \(p\geq 1\,,\) satisfy the assumptions of Theorem \ref{thelambdaphi}. 
    Also note that as it is the case for $\psi$ in Theorem \ref{thelambdapsi} also \(\varphi\) may attain negative values.
    \item The second part of the proof of Theorem \ref{thelambdaphi}\eqref{thelambda13} can also be established along the lines of the proof of the second part of Theorem \ref{thelambdapsi}\eqref{thelambdapsi3}. However, we opted for 
    a more constructive proof which provides additional insight on 
    properties of $\Lambda_\varphi$. Instead of the denseness argument 
    here we use the fact that for independent \(X_1\) and \(X_2\) 
    with \(X_1\eqd X_2\eqd X\), the random variable \(F_{Y|X=X_1}(y)-F_{Y|X=X_2}(y)\) obtained from \eqref{eqnocoun} is bounded on \([-1,1]\) for every \(y\in \R\) and the support of the extremal elements w.r.t. convex order is \(\PP^Y\)-almost surely the set \(\{-1,0,1\}\,.\)
\end{enumerate}
\end{remark}

\newpage
\subsection{Additional properties of \(\Lambda_\varphi\)}

Convexity of the function \(\varphi\) yields the so-called 
\emph{data processing inequality} which states that applying a transformation 
to the explanatory variable $X$ can not increase the extent of dependence, see \cite{cover2006} for a detailed discussion. Note that strict convexity in \(0\) is not required for the following results.

\begin{proposition}[Data processing inequality]~~\\
Assume that \(\varphi\) is convex on \([-1,1]\).
Then \(\Lambda_\varphi\) defined according to eq. \eqref{eq:Lamda.delta.phi} fulfills the data processing inequality, 
i.e.,
\begin{eqnarray*}
  \Lambda_\varphi(Y|Z) \leq \Lambda_\varphi(Y|X)
\end{eqnarray*}
for all $X,Y,Z$ such that $Y$ and $Z$ are conditionally independent given $X$.
In particular, 
\begin{eqnarray*} 
  \Lambda_\varphi(Y|h(X)) \leq \Lambda_\varphi(Y|X)
\end{eqnarray*}
for all measurable functions $h: \R \rightarrow \R$. 
\end{proposition}
\begin{proof}
%Since \(\varphi\) is convex and \(\alpha_\varphi\) positive and finite we have 
%\(\varphi(0)/\alpha_\varphi\leq \Lambda_\varphi(Y|Z) \leq 1\), implying 
%that \(\Lambda_\varphi(Y|Z)\) is well-defined. 
Assume that $Y$ and $Z$ are conditionally independent given $X$.
Proceeding as in the proof of the second assertion of Lemma 3.5 in \cite{fgwt2021} 
and using disintegration it follows that for every fixed 
$y \in \mathbb{R}$ we have  
\begin{align*}
    F_{Y|Z=z}(y)
    = \int_{\mathbb{R}} F_{Y|Z=z,X=x}(y) \de \PP^{X|Z=z}(x) 
    = \int_{\mathbb{R}} F_{Y|X=x}(y) \de \PP^{X|Z=z}(x) 
\end{align*}
for $\PP^{Z}$-almost all $z \in \mathbb{R}$. Hence, using
convexity of \(\varphi\), Jensen's inequality, and disintegration once more yields
\begin{align*}
    &\int_{\R^2}\int_{\R} \varphi(F_{Y|Z=z_1}(y)-F_{Y|Z=z_2}(y)) \de \PP^Y(y) \de (\PP^{Z}\otimes \PP^{Z})(z_1,z_2)
    \\
    & = \int_{\R^2}\int_{\R} \varphi \left( \int_{\mathbb{R}} F_{Y|X=x_1}(y) \de \PP^{X|Z=z_1}(x_1) - \int_{\mathbb{R}} F_{Y|X=x_2}(y) \de \PP^{X|Z=z_2}(x_2)  \right) \de \PP^Y(y) \de (\PP^{Z} \otimes \PP^{Z})(z_1,z_2)
    \\
    & = \int_{\R^2}\int_{\R} \varphi \left( \int_{\mathbb{R}^2} F_{Y|X=x_1}(y) - F_{Y|X=x_2}(y) \de (\PP^{X|Z=z_1} \otimes \PP^{X|Z=z_2}) (x_1,x_2)  \right) \de \PP^Y(y) \de (\PP^{Z}\otimes \PP^{Z})(z_1,z_2)
    \\
    & \leq \int_{\R^2} \int_{\mathbb{R}^2} \int_{\R} \varphi( F_{Y|X=x_1}(y) - F_{Y|X=x_2}(y) ) \de \PP^Y(y) \de (\PP^{X|Z=z_1} \otimes \PP^{X|Z=z_2}) (x_1,x_2) \de (\PP^{Z}\otimes \PP^{Z})(z_1,z_2)
    \\
    & = \int_{\R^2} \int_{\mathbb{R}^2} \int_{\R} \varphi( F_{Y|X=x_1}(y) - F_{Y|X=x_2}(y) ) \de \PP^Y(y) \de \PP^{(X,Z)} (x_1,z_1) \de \PP^{(X,Z)} (x_2,z_2)
    \\
    &  = \int_{\R^2} \int_{\R} \varphi( F_{Y|X=x_1}(y) - F_{Y|X=x_2}(y) ) \de \PP^Y(y) \de (\PP^{X} \otimes \PP^{X}) (x_1,x_2)\,,
\end{align*}
which completes the proof.
\end{proof}

As a consequence of the data processing inequality, \(\Lambda_\varphi\) satisfies the property of \emph{self-equitability} introduced in \cite{kinney2014} (see also \cite{ding2017}) which roughly states that for $(X,Y)$ in a regression 
setting \(Y = f (X) + \varepsilon\,,\) the function $\Lambda_\varphi$ should only
depend on the distribution of the noise \(\varepsilon\) and not on the specific form of the function \(f\), see \cite{strothmann2022}.
%The data processing inequality implies an interesting invariance property for \(\Lambda_\varphi\).
%The first one is the so-called \emph{self-equitability} introduced in \cite{kinney2014} (see also \cite{ding2017}).
%According to \cite{reshef2011} self-equitability states that
%``the statistic should give similar scores to equally noisy relationships of different types''.
%In a regression setting \(Y = f (X) + \varepsilon\) it means that
%the measure of predictability \(T^q\) ``depends only on the strength of the noise \(\boldsymbol{\varepsilon}\) and not on the specific form of \(f\)'' \cite{strothmann2022}.

\begin{corollary}[Self-equitability]~~\\
Assume that \(\varphi\) is convex on \([-1,1]\,.\) 
Then \(\Lambda_\varphi\) defined according to \eqref{eq:Lamda.delta.phi} is 
\emph{self-equitable}, 
i.e.,
\begin{eqnarray*}
  \Lambda_\varphi(Y|h(X)) = \Lambda_\varphi(Y|X)
\end{eqnarray*}
for all $(X,Y)$ and all measurable functions $h$ such that $Y$ and $X$ are conditionally independent given $h(X)$.
\end{corollary}

The following result provides additional useful invariance properties for \(\Lambda_\varphi\). Note that for these properties to hold
we do not even need \(\varphi\) to be convex.

\begin{proposition}[Invariances]\label{thelambda2}~\\
    \(\Lambda_\varphi\) satisfies the invariance properties (P\ref{prop.INVX}), (P\ref{prop.INVY}), and (P\ref{prop.lawINV}), i.e.,
    \begin{enumerate}[(i)]\setcounter{enumi}{0}
        \item \label{thelambda14} \(\Lambda_\varphi(Y|X) = \Lambda_\varphi(Y|h(X))\) for all bijective functions \(h: \R \rightarrow \R\),
        \item \label{thelambda15} \(\Lambda_\varphi(Y|X) = \Lambda_\varphi(g(Y)|X)\) for all strictly increasing functions \(g: \R \rightarrow \R\),
        \item \label{thelambda16} \(\Lambda_\varphi(Y|X) = \Lambda_\varphi(Y'|X')\) whenever \((X,Y)\eqd (X',Y')\),
    \end{enumerate}
\end{proposition}

\begin{proof}
    \eqref{thelambda14}: Let \((X',Y')\) be an independent copy of \((X,Y)\). Since \(h\) is bijective, the \(\sigma\)-algebras generated by \(X\) and \(h(X)\) coincide. 
    Using change of coordinates it therefore follows that
    \begin{align*}
        \Lambda_\varphi(Y|X) &= \int_\R \int_\Omega \varphi\left(\E[\1_{\{Y\leq y\}}|X](\omega) - \E[\1_{\{Y'\leq y\}}|X'](\omega)\right) \de \PP(\omega)\de \PP^Y(y)\\
        &= \int_\R \int_\Omega \varphi\left(\E[\1_{\{Y\leq y\}}|h(X)](\omega) - \E[\1_{\{Y'\leq y\}}|h(X')](\omega)\right) \de \PP(\omega)\de \PP^{Y}(y) = \Lambda_\varphi(Y|h(X))\,.
    \end{align*}
    \eqref{thelambda15}: Since there exists some measurable function \(k\colon \R^2 \to [0,1]\) such that \(F_{Y|X=x}(y) = k(x,F_Y(y))\) for \(\PP^X\)-almost all \(x\in \R\) and for all \(y\in \R\,,\) see \cite[Theorem 2.2]{Ansari-2021}, 
     using change of coordinates and that \(g\) is strictly increasing it follows that
    ($\lambda$ denoting the Lebesgue measure, $F^-$ the quasi-inverse of a distribution 
    function $F$).
    \begin{align*}
        \Lambda_\varphi(Y|X) &= \int_{\R^2} \int_{(0,1)} \varphi\left(k\left(x_1,F_Y\circ F_Y^{-}(u)\right)-k(x_2,F_Y\circ F_Y^{-}(u))\right) \de \lambda(u) \de \PP^X\otimes \PP^X(x_1,x_2)\\
        &= \int_{\R^2} \int_{(0,1)} \varphi\left(k\left(x_1,F_{g(Y)}\circ F_{g(Y)}^{-}(u)\right)-k\left(x_2,F_{g(Y)}\circ F_{g(Y)}^{-}(u)\right)\right) \de \lambda(u) \de \PP^X\otimes \PP^X(x_1,x_2)\\
        &= \Lambda_\varphi(g(Y)|X)).
    \end{align*}
      Statement \eqref{thelambda16} is trivial.
\end{proof}

Considering that \(\Lambda(Y|X)\) according to Theorem \ref{thelambda2}\eqref{thelambda14} is invariant with respect to bijective transformations of \(X\), in order to establish some general ordering results 
for \(\Lambda_\varphi\) we need to work with rearrangement invariant orders. 
To this end, for integrable functions \(f,g\colon (0,1)\to \R\) with \(\int_0^1 f(t) \de \lambda(t) = \int_0^1 g(t) \de \lambda(t)\) consider the Schur order $\prec_S$, defined by
\begin{align}
   f\prec_S g \quad : \,\Longleftrightarrow \quad  \int_0^x f^*(t) \de t \leq \int_0^x g^*(t) \de t \quad \text{for all } x\in (0,1)\,,
\end{align}
where \(h^*\) denotes decreasing rearrangement of a function \(h\colon (0,1)\to \R\,,\) i.e., the $\lambda$-almost everywhere unique decreasing function $h^*$ 
such that \(\lambda(\{h^*\leq t\}) = \lambda(\{h\leq t\})\) for all \(t\in \R\,.\)
Denote by \(q_Z(u) :=\inf\{z \in \R: \mid F_Z(z)\geq u\} =F_Z^-(u)\) the quasi-inverse of 
the distribution function $F_Z$ of $Z$. 
%Denote by \(\leq_S\) the Schur order for integrable functions on \((0,1)\,.\)
Then, for \(Y\) and \(Y'\) having continuous distribution functions, the Schur order \((Y|X)\leq_S (Y'|X')\) for conditional distributions is defined by
\begin{align}
 \E\left[\1_{\{Y\leq q_Y(v)\}}\mid X=q_X(\cdot)\right] \leq_S \E\left[\1_{\{Y'\leq q_{Y'}(v)\}}\mid X' = q_{X'}(\cdot)\right] ~~~\text{for all } v\in (0,1)\,,
\end{align}
see \cite{Ansari-Fuchs-2022} as well as \cite{Ansari-2019,Ansari-2021,Shaked-2013,strothmann2022}.
Hence, the Schur-order for conditional distributions compares the variability of conditional distribution functions in the conditioning variable.

For bivariate random vectors \((U,V)\) and \((U',V')\,,\) denote by \((U,V)\leq_{lo}(U',V')\) the lower orthant order, i.e., the pointwise comparison of the associated distribution functions. Further, the bivariate random vector \((U,V)\) is said to be stochastically increasing (SI) if the conditional distribution \(V|U=u\) is stochastically increasing in \(u\,,\) i.e., \(\E[f(V)|U=u]\) is increasing in \(u\) for all increasing functions \(f\) such that the expectations exist.
If \((U',V')\) is SI and if \(U\eqd U'\) and \(V\eqd V'\,,\) then \((V|U)\leq_{S}(V'|U')\) implies \((U,V)\leq_{lo} (U',V')\,.\)  If additionally \((U,V)\) is SI, then \((V|U)\leq_{S}(V'|U')\) and \((U,V)\leq_{lo} (U',V')\) are equivalent, see \cite[Lemma 3.16]{Ansari-2021}.\\

The following result provides simple conditions for monotonicity of \(\Lambda_\varphi\,.\)

\begin{proposition}[Monotonicity]\label{thelambda3}~\\
    Suppose that \((Y|X)\leq_S (Y'|X')\) and that $Y,Y'$ have continuous distribution function. Then \(\Lambda_\varphi(Y|X)\leq \Lambda_\varphi(Y'|X')\) for all convex functions \(\varphi\) such that the integrals exist.
\end{proposition}

\begin{proof}
    By the Hardy-Littlewood-Polya theorem (see, e.g., \cite[Theorem 3.21]{Ru-2013} we have that \((Y|X)\leq_S (Y'|X')\) is equivalent to
    \begin{align*}
        \int_0^1 \phi(F_{Y|X=q_X(u)}(q_Y(v))) \de \lambda(u) \leq \int_0^1 \phi(F_{Y'|X'=q_{X'}(u)}(q_{Y'}(v))) \de \lambda(u) \quad 
    \end{align*}
    for all $v$ and for all convex $\phi$ such that the integrals exist. Since the functions \(z \mapsto \varphi( z-c )\) and \(z \mapsto \varphi( c-z )\) are convex whenever \(\varphi\) is convex, using Fubini's theorem and change of coordinates
    repeatedly it follows that 
    \begin{align*}
    &\int_{\R^2}\int_{\R} \varphi(F_{Y|X=x_1}(y)-F_{Y|X=x_2}(y)) \de \PP^Y(y) \de (\PP^{X}\otimes \PP^{X})(x_1,x_2)\\
    &= \int_{(0,1)}\int_{(0,1)} \int_{(0,1)} \varphi\left(F_{Y|X=q_X(u_1)}(q_Y(v)) - F_{Y|X=q_X(u_2)}(q_Y(v)) \right) \de \lambda(u_1) \de \lambda(v) \de \lambda(u_2)  \\
    &\leq \int_{(0,1)} \int_{(0,1)} \int_{(0,1)} \varphi\left(F_{Y'|X'=q_{X'}(u_1)}(q_{Y'}(v)) - F_{Y'|X'=q_{X'}(u_2)}(q_{Y'}(v)) \right) \de \lambda(v) 
    \de \lambda(u_1) \de \lambda(u_2) \\
        &= \int_{\R^2}\int_{\R} \varphi(F_{Y'|X'=x_1}(y)-F_{Y'|X'=x_2}(y)) \de \PP^{Y'}(y) \de (\PP^{X'}\otimes \PP^{X'})(x_1,x_2)
    \end{align*}
    which implies \(\Lambda_\varphi(Y|X)\leq \Lambda_\varphi(Y'|X')\,.\)
\end{proof}

\begin{remark}
\begin{enumerate}[(a)]
    \item Minimal and maximal elements with respect to the Schur order for conditional distributions characterize independence and complete dependence, respectively, see \cite[Theorem 3.5]{Ansari-Fuchs-2022}. Hence, as a consequence of Theorem \ref{thelambda3}, if \(Y\) is completely depend on \(X\) and if \(\varphi\) is convex, then \(\Lambda_\varphi(Y|X) = 1\,,\) using that \(\alpha_\varphi\in (0,\infty)\,.\) If additionally \(\varphi\) is strictly convex in \(0\) and if \(\varphi(0)=0\,,\) then Theorem \ref{thelambdaphi} provides the reverse implication, i.e., that \(\Lambda_\varphi(Y|X)= 1\) implies that \(Y=f(X)\) almost surely.
    \item Various well-known (sub-)families of bivariate distributions are stochastically increasing and increasing with respect to the lower orthant order (implying that
    the associated conditional distributions are increasing with respect to the Schur order for conditional distributions). Examples are the bivariate normal distribution with respect to the non-negative correlation parameter, several Archimedean copula families such as the Clayton, Gumbel or Frank family, or various extreme value copula families, see, e.g., \cite[Chapter 4]{Nelsen-2006}, \cite{Mueller-Scarsini-2005}, and \cite[Examples 3.18 and 3.19]{Ansari-2021}.
\end{enumerate}
    
\end{remark}

\newpage
\subsection{Continuous setting and $L^p$-version} \label{secLp}
From now on we will assume that $(X,Y)$ has continuous bivariate 
distribution function since continuity is necessary for having uniqueness in 
 Sklar's theorem (see \cite{Nelsen-2006}) and both for deriving 
continuity results of the functional \(\Lambda_\varphi\) and for the the estimation procedure studied in the next section. 
Furthermore, in the sequel (following \cite{JGT}) we will write conditional distributions as Markov kernels, i.e., $\PP^{Y|X=x}(\cdot)=K(x,\cdot)$, let $C$ denote the (unique) copula underlying 
$(X,Y)$ and $K_C(\cdot,\cdot)$ as the Markov kernel of $C$.
Notice that according to \cite[Lemma 1]{MFFT} (also see  \cite[Theorem 2.2]{Ansari-2021}), 
starting with a Markov kernel $K_C(\cdot,\cdot)$ and setting 
$$K(x,(-\infty,y]):=K_C(F_X(x),[0,F_Y(y)])$$
yields a (version of the) Markov kernel of $(X,Y)$. 
In the continuous setting, the functional $\Lambda_\varphi$ only depends on the 
underlying copula $C$ since for every continuous \(\varphi\colon [-1,1]\to \R\)
using change of coordinates we have 
\begin{align}\label{eq:trafo.to.C}
        \Lambda_{\varphi}(Y|X) 
        &= \int_{\R^2}\int_{\R} \varphi\big(F_{Y|X=x_1}(y)-F_{Y|X=x_2}(y)\big) \de \PP^{X}\otimes \PP^{X}(x_1,x_2) \de \PP^{Y}(y) \nonumber\\
        &= \int_{\R^2}\int_{\R} \varphi\big(K_C(F_X(x_1),[0,F_Y(y)])-K_C(F_X(x_2),[0,F_Y(x)])\big) \de \PP^{X}\otimes \PP^{X}(x_1,x_2) \de \PP^{Y}(y)\nonumber\\
        &= \int_0^1\int_0^1\int_0^1 \varphi\big(K_{C}(u_1,[0,v]) - K_C(u_2,[0,v]\big)
        \de \lambda(v)\de \lambda(u_1)\de \lambda(u_2).
\end{align}        

\noindent Following \cite{KFT} we say that a sequence a sequence \((C_n)_{n\in \N}\) of bivariate copulas \emph{converges weakly conditional} to a bivariate copula \(C\) and 
write \(C_n\xrightarrow{wcc} C\,,\) 
if there exists some Borel set $\Gamma \subseteq [0,1]$
with $\lambda(\Gamma)=1$ such that the probability measures
$(K_{C_n}(u,\cdot))_{n \in \mathbb{N}}$ converges weakly to the probability measure $K_C(u,\cdot)$ for every \(x\in \Gamma\). 
The following result provides sufficient conditions for continuity of the functional \(\Lambda_\varphi\) and is straightforward to prove using eq. \eqref{eq:trafo.to.C} 
and dominated convergence. 
\begin{proposition}[Continuity]
    Let \((X,Y),(X_1,Y_1),(X_2,Y_2),\ldots\) be bivariate random vectors with continuous distribution functions and corresponding copulas $A,A_1,A_2,\ldots$ and 
    let $\varphi\colon [-1,1] \to \R$ be continuous. Then
    $A_n\xrightarrow{wcc} A$ implies convergence of $(\Lambda_\varphi(Y_n|X_n))_{n \in \mathbb{N}}$ to $\Lambda_\varphi(Y|X)$.
   \end{proposition}

\begin{example}\label{ex:mono}
A particular class of dependence measures fulfilling all of the afore-mentioned properties
is obtained when considering 
\(\varphi(x)=|x|^p\) for some \(p\geq 1\). In this case \(\Lambda_\varphi\) corresponds to the \(L^p\)-distance of conditional distribution functions. In fact, in this case 
using eq. \eqref{eq:trafo.to.C} we have 
\begin{comment}
Then eq. (\ref{eq:Lamda.allg}) boils down to the 
copula setting: In fact, using \cite[Lemma 1]{MFFT} we have 
$$
F_{Y\vert X=x}(y) = \lim_{z\downarrow y} \partial_1 A(G_1(x),G_2(z)) = K_A\left(G_1(x),[0,G_2(y)]\right)
%K_H(x,(-\infty,y])= K_A\left(G_1(x),[0,G_2(y)]\right) =: F_A^{G_1(x)} \circ G_2(y),
$$
for \(P^X\)-almost all \(x\in \R\,.\)
Hence applying the probability integral transform and change of coordinates yields
\begin{align*}
\Lambda_\Delta(Y|X) = & \int_{\mathbb{R}^2} \Delta(F_{Y\vert X=x_1},F_{Y\vert X=x_2}) 
  \de(\mathbb{P}^X \otimes \mathbb{P}^X)(x_1,x_2) \\
  %= & \int_{\mathbb{R}^2} d\left(\lim_{z\downarrow y} \partial_1 A(G_1(x_1),G_2(z)), \lim_{z\downarrow y} \partial_1 A(G_1(x_2),G_2(z))\right) \,
   %\de (\mathbb{P}^X \otimes \mathbb{P}^X)(x_1,x_2)  \\
   = & \int_{\mathbb{R}^2} \Delta \left( K_A\left(u_1,[0,G_2(y)]\right),K_A\left(u_2,[0,G_2(y)]\right)\right) \,
   \de \lambda^2(u_1,u_2)\,.
\end{align*}
%\textcolor{blue}{Einfachere Notation möglich?}
\noindent For the special case of the \(L^p\)-metric
$$
\Delta(F_1,F_2) = \int_{\mathbb{R}^2} \vert F_1(y)- F_2(y) \vert^p \de \mathbb{P}^Y(y), \quad p \geq 1\,,
$$
\end{comment}
\begin{align}\label{eq:Lamda.delta.p}
\Lambda_p(Y|X) &:=  3 \int_{[0,1]^3} \vert K_A(u_1,[0,z]) -  K_A(u_2,[0,z]) \vert^p \, \de \lambda^3(u_1,u_2,z) = \Lambda_p(A)\,,
\end{align}
where the normalizing constant according is given by 
\begin{align}
    \alpha_p:=\alpha_\varphi= \int_{(0,1)^3} |\1_{\{u_1\leq v\}} - \1_{\{u_2\leq v\}}|^p \de \lambda^3(u_1,u_2,v).
\end{align}
Since the normalizing constant does not depend on \(p\) it follows that the mapping
$p \mapsto \Lambda_p(Y|X)$  is decreasing in $p \in [1,\infty)$. \\
For \(p=2\) the functional \(\Lambda_p\) compares two conditional distributions as 
in \eqref{eq:Lamda.allg} and coincides with Chatterjee's correlation coefficient 
(see Example \ref{Ex:Chatterjee}) which compares conditional and unconditional distributions as in \eqref{eq:cond-uncond}. 
For \(p=1\,,\) however, the following examples illustrates that the comparison of two conditional distributions as in \eqref{eq:Lamda.allg} only coincides for some specific families with the \(L^1\)-analogue of Chatterjee's coefficient of correlation.
\end{example}
\begin{example}[FGM and Fr{\'e}chet copulas]\label{Ex.FGM.Frechet}~~
The functional \(\zeta_1(Y|X)\) defined by
\begin{align}
    \zeta_1(Y|X) := \frac{1}{3}\int_0^1 | F_{Y|X=x}(y)-F_Y(y)| \de \PP^Y(y) \de \PP^X(x)
\end{align}
is a dependence measure introduced and studied in \cite{wt2011}.
\begin{enumerate}[(a)]
\item For $\alpha \in [0,1]$ the copula 
\begin{align*}
  C_{\alpha} (u,v)
  := uv + \alpha \, u(1-u)v(1-v)
\end{align*}
is a member of the Farlie-Gumbel-Morgenstern (FGM) family. 
Let \((X,Y)\) be a random vector with continuous distribution function and underlying 
copula \(C_{\alpha}\). Then straightforward calculations yield
\begin{align*}
    \Lambda_1(Y|X) 
    = \frac{\alpha}{3} 
    \qquad \textrm{ and } \qquad
    \zeta_1(Y|X) 
    = \frac{\alpha}{4}, 
\end{align*}
so \(\Lambda_1(Y|X) = \tfrac{4}{3} \, \zeta_1(Y|X) \neq \zeta_1(Y|X)\).

\item For $\alpha \in [0,1]$ the copula 
\begin{align*}
  C_{\alpha} (u,v)
  := \alpha \, M(u,v) + (1-\alpha) \, \Pi(u,v)
\end{align*}
is a member of the Fr{\'e}chet copula family. In this case 
for continuous \((X,Y)\) with underlying copula \(C_{\alpha}\) we get 
\begin{align*}
    \Lambda_1(Y|X) 
    %& = 3 \int_{\mathbb{R}^2} \int_{\mathbb{R}} \vert F_{Y|X=x_1}(y) - F_{Y|X=x_2}(y) \vert \, \de \mathbb{P}^Y(y) \de(\mathbb{P}^X \otimes \mathbb{P}^X)(x_1,x_2)
    %\\
    %& 
    = \alpha 
    %\\
    %& = 3 \int_{\mathbb{R}} \int_{\mathbb{R}} \vert F_{Y|X=x}(y) - F_{Y}(y) \vert \, \de \mathbb{P}^Y(y) \de \mathbb{P}^X (x)
    %\\
    %& 
    = \zeta_1(Y|X) \,.
\end{align*}
\end{enumerate}
\end{example}

\section{Estimating the functionals}\label{sec_estimation}
In this section, we propose a copula-based so-called checkerboard estimator 
as plug-in estimator for the functionals \(\Lambda_\psi\) in \eqref{eq:Lamda.delta2.phi} and \(\Lambda_\varphi\) in \eqref{eq:Lamda.delta.phi} and start with recalling some basic facts.\\
For every bivariate copula $A$ and $N \in \N$ the $N$-checkerboard 
approximation of $A$ (see \cite{LMT,JGT}) will be denoted by $\Ch_N(A)$.
As a consequence of \cite[Corollary 3.2]{LMT}, %\textcolor{blue}{hier braucht's das Argument mit abzählbar vielen \(y\)...}
 for every copula $A$ the sequence $(\Ch_N(A))_{N \in \mathbb{N}}$ converges
weakly conditional to $A$. 
To assure well-definedness we will refer to the bilinear interpolation of the 
subcopula resulting from Sklar's theorem by considering the empirical distribution functions as the empirical copula, see, e.g., \cite{Genest-2017,fgwt2021}.\\
Plugging in the empirical copula we get the following consistency result.

%\textcolor{blue}{eher Lemma oder Proposition?}
\begin{lemma}\label{thm:checkerboard.wcc}
Let $(X_1, Y_1), (X_2, Y_2),\ldots$  be a random sample from $(X,Y)$ and assume that $(X,Y)$ has continuous 
distribution function $H$, underlying copula $A$, and that $E_n$ denotes the empirical copula 
(of the first $n$ observations). Then setting $N(n) :=\lfloor{n}^s\rfloor$ for some 
$s \in (0,\frac{1}{2})$ the sequence $(\Ch_{N(n)}(E_n))_{n \in \mathbb{N}}$ converges weakly conditional to $A$
with probability $1$.
\end{lemma}

\begin{proof}
Choose $\tilde{\Gamma} \in \mathcal{B}([0,1])$ with $\lambda(\tilde{\Gamma})=1$ in such a way that for every 
$x \in \tilde{\Gamma}$ we have
that $(K_{\Ch_N(A)}(x,\cdot))_{N \in \mathbb{N}}$ converges weakly to $K_A(x,\cdot)$. 
Setting $\Gamma := \tilde{\Gamma} \setminus \bigcup_{N=2}^\infty\{0,\frac{1}{N},\frac{2}{N},\ldots,\frac{N-1}{N},1\}$ obviously yields
$\lambda(\Gamma)=1$. \\
Fix $x \in \Gamma$ and consider some $y \in (0,1)$ with
\begin{equation}\label{eq:temp2}
\lim_{N \rightarrow \infty} K_{\Ch_N(A)}(x,[0,y]) = K_A(x,[0,y]).
\end{equation}
Plugging in the empirical copula, setting $I_n:=\vert K_{\Ch_{N(n)}(E_n)}(x,[0,y]) - K_A(x,[0,y]) \vert $,
 and using the triangle inequality yields  
\begin{eqnarray*}
I_n &\leq& 
\vert K_{\Ch_{N(n)}(E_n)}(x,[0,y]) - K_{\Ch_{N(n)}(A)}(x,[0,y]) \vert + \vert K_{\Ch_{N(n)}(A)}(x,[0,y]) - K_{A}(x,[0,y]) \vert
\end{eqnarray*}
The second summand converges to $0$ for $n \rightarrow \infty$ due to \eqref{eq:temp2}. In order to show that the first summand converges to \(0\) almost surely, denote by 
\(i_N(x)\,,\) \(N\geq 2\,,\) the unique integer in \(\{1,\ldots,N\}\) fulfilling 
$x \in (\frac{i_N(x)-1}{N},\frac{i_N(x)}{N})$. \\
Since $\Ch_{N(n)}(E_n)$ and
$\Ch_{N(n)}(A)$ are checkerboard copulas, we have
\begin{eqnarray*}
K_{\Ch_{N(n)}(E_n)}(x,[0,y]) &=& N(n) \left(E_n\left(\frac{i_{N(n)}(x)}{N(n)},y\right) 
-  E_n\left(\frac{i_{N(n)}(x)-1}{N(n)},y\right)\right) \\ 
K_{\Ch_{N(n)}(A)}(x,[0,y]) &=& N(n) \left(A\left(\frac{i_{N(n)}(x)}{N(n)},y\right) 
-  A\left(\frac{i_{N(n)}(x)-1}{N(n)},y\right)\right).
\end{eqnarray*}
This, however, implies
\begin{eqnarray*}
\limsup_{n \rightarrow \infty}\vert K_{\Ch_{N(n)}(E_n)}(x,[0,y]) - K_{\Ch_{N(n)}(A)}(x,[0,y]) \vert \leq 2 
\limsup_{n \rightarrow \infty} N(n) \, d_\infty(E_n,A)=0\,,
\end{eqnarray*}
\(\PP\)-almost surely, whereby the last identity is a consequence of the definition of 
\(N(n)\) and the fact that according to \cite[Lemma 1]{JSV} the empirical copula $E_n$ satisfies
\begin{equation}
d_\infty(E_n,A)= O\left(\sqrt{\frac{\log(\log(n))}{n}} \right)
\end{equation}
with probability $1$. Altogether it follows that $\lim_{n \rightarrow \infty} I_n=0$
almost surely. \\
Since all but at most countably many $y \in (0,1)$ fulfill  (\ref{eq:temp2}) 
and weak convergence is equivalent to convergence on a dense set, we have shown 
weak convergence of $(K_{CB_{N(n)}(E_n)}(x,\cdot))_{n \in \mathbb{N}}$ to 
$(K_{A}(x,\cdot))_{n \in \mathbb{N}}$ for 
every $x \in \Gamma$, which completes the proof.
\end{proof}

Having this, the following main result of this section is straightforward to prove.
%Having the above result it is straightforward to derive a strongly consistent estimator for the $L^p$-version without any smoothness assumptions on the underlying copula:
\begin{theorem}[Strongly consistent estimator for \(\Lambda_\varphi(A)\)]\label{theestphi}~\\
Let $(X_1, Y_1), (X_2, Y_2),\ldots$  be a random sample from $(X,Y)$, assume 
that $(X,Y)$ has continuous distribution function, and let \(A\) denote 
the copula of \((X,Y)\). Furthermore choose $s$ and the resolution $N(n)$ as in Lemma \ref{thm:checkerboard.wcc}.
If \(\varphi\) is continuous, then $\Lambda_{\varphi}(\Ch_{N(n)}(E_n))$ is a strongly consistent estimator for $\Lambda_{\varphi}(A)$, i.e., 
with probability $1$, we have
$$
\lim_{n \rightarrow \infty} \Lambda_{\varphi}(\Ch_{N(n)}(E_n))=\Lambda_{\varphi}(A).
$$ 
\end{theorem} 
\begin{proof}
According to (the proof of) Lemma \ref{thm:checkerboard.wcc} we have that 
$(K_{CB_{N(n)}}(u,\cdot))_{n \in \mathbb{N}}$
converges weakly to $K_A(u,\cdot)$ for every $u \in \Gamma$. 
As a direct consequence, for $u_1,u_2 \in \Gamma$ it follows that
\begin{eqnarray*}
\lim_{n \rightarrow \infty} \vert K_{CB_{N(n)(E_n)}}(u_1,[0,y])) - K_{CB_{N(n)(E_n)}}(u_2,[0,y])) \vert
&=&\vert K_{A}(u_1,[0,y])) - K_{A}(u_2,[0,y])) \vert,
\end{eqnarray*}
hence, considering $\lambda_2(\Gamma \times \Gamma)=1$, using eq. (\ref{eq:Lamda.delta.p}), and applying dominated convergence we have
\begin{eqnarray*}
 &&\lim_{n \rightarrow \infty}  \int_{[0,1]^2} \int_{[0,1]} \vert K_{\Ch_{N(n)(E_n)}}(u_1,[0,y])) - 
   K_{\Ch_{N(n)(E_n)}}(u_2,[0,y])) \vert ^p d\lambda(y)\, d\lambda_2(u_1,u_2) \\
   &&= \int_{[0,1]^2} \int_{[0,1]} \vert K_{A}(u_1,[0,y])) - 
   K_{A}(u_2,[0,y])) \vert ^p d\lambda(y)\, d\lambda_2(u_1,u_2)= \Lambda_{p}(A).
\end{eqnarray*}
\end{proof}
\begin{remark}
    Notice that according to Theorem \ref{theestphi} strong consistency holds in full 
    generality, i.e., without any smoothness restrictions for the copula $A$. 
    We also conjecture that the estimator $\Lambda_{\varphi}(\Ch_{N(n)}(E_n))$ 
    is asymptotically normal, a proof, however, seems still out of reach, 
    see, e.g., \cite{strothmann2022}.
\end{remark}

\begin{remark}
    The checkerboard estimator can also be used for estimating \(\Lambda_\Delta(Y|X)\) in the more general case where the distance function \(\Delta\) is at least component-wise continuous with respect to weak convergence and invariant with respect to strictly increasing transformations of \(Y\).
\end{remark}

\clearpage
\section{Simulation Study and Real Data Example}\label{sec:sim}
In this section we first study the performance of the checkerboard estimator \(\Lambda_{\varphi}(\Ch_{N(n)}(E_n))\) from Theorem \ref{theestphi} 
and then quickly discuss a real data example.
\subsection{Simulation study}
We consider the continuous setting as studied in the previous section focus on the  following choices for the convex function \(\varphi\) - notice that that $\varphi$
is convex and strictly convex in $0=\varphi(0)$, so $\Lambda_\varphi$
is a dependence measure:
\begin{enumerate}[(i)]
    \item \label{phityp1} \(\varphi(x) = |x|^p\) for \(p\in \{1,2,3\}\,,\)
    \item \label{phityp2} \(\varphi(x) = e^{cx}-1\) for \(c\in \{1/5, 1, 5\}\,,\)
    \item \label{phityp3} \(\varphi(x) = e^{|cx|}-1\) for \(c\in \{1/5,1,5\}\,.\)
\end{enumerate}
%present the results of a simulation study testing the previously introduced methods. We focus on the case where $\Lambda_{\Delta}(A)$ is of the form~\ref{eq_special_delta}. Specifically we look at
\begin{comment}
\begin{equation*}
    \varphi(x) \in \{|x|^p, e^{|cx|} - 1, e^{cx} - 1\}
\end{equation*}
with
\begin{align*}
    p \in \{1, 2, 3\} \text{ and } c \in \{\frac{1}{5}, 1, 5\}.
\end{align*}
\end{comment}
As dependence structures we chose copulas \(A\) from the Marshall-Olkin (MO) family, defined by
\begin{equation*}
    M_{\alpha, \beta}(u, v) := \begin{cases}
        u^{1 - \alpha} v\,, &\text{if}\quad u^{\alpha} \geq v^{\beta}\,,\\
        uv^{1 - \beta}\,, &\text{if}\quad u^{\alpha} < v^{\beta}\,,
    \end{cases}
\end{equation*}
with \(\alpha,\beta\in [0,1]\) and consider  
\begin{equation*}
    (\alpha, \beta) \in \{(1, 0), (1, 1), (0.2, 0.7), (0.3, 1)\}.
\end{equation*}
Figures \ref{MO_1_0.pdf} to \ref{MO_0.3_1.pdf} illustrate the behavior of the estimator 
\(\Lambda_{\varphi}(\Ch_{N(n)}(E_n))\) from Theorem \ref{theestphi} 
for each of the chosen copulas \(A\) and sample sizes given by
\begin{equation*}
    n \in \{10, 50, 100, 500, 1.000, 5.000, 10.000\},
\end{equation*}
where each scenario was simulated \(1.000\) times.
Estimation of the empirical checkerboard copula was performed using the command \textit{ecbc} from the R package \textit{qad} with the default of $s = 1/2$ (see \cite[Remark 3.14]{JGT} for a detailed discussion of this choice).
Figures \ref{MO_1_0.pdf} to \ref{MO_0.3_1.pdf} can be interpreted as follows:  

\begin{itemize}
    \item the measures with convex function \(\varphi\) of type \eqref{phityp2} perform best in all cases in terms of speed of convergence. 
    One possible explanation of this observation is asymmetry of 
    \(\varphi(x) = e^{cx}-1\,:\) The number of distinct values over which the integral runs is potentially twice as large and might therefore speed up convergence. 
    As a drawback, however, for these functions, \(\Lambda_\varphi(Y|X)\) often attains small values, see Figures \ref{MO_0.2_0.7.pdf} and \ref{MO_0.3_1.pdf}, which,
    however should not be confused with independence as considered in Figure \ref{MO_1_0.pdf}.
    \item \(\Lambda_\varphi(Y|X)\) is decreasing in the parameters \(p\) and \(c\) (compare with Example \ref{ex:mono}). Consequently, in the case of complete dependence, the measures with parameter \(p=1\) and \(c=1/5\) perform best because  deviations are less penalized by the shape of the function \(\varphi\) than 
    by the other cases. In contrast to that, if the parameters \(p\) and \(c\) are large, then the measure \(\Lambda_\varphi\) is more sensitive for detecting independence.
    \item the \(L^2\)-version \(\varphi(x)=x^2\) as well as the asymmetric version \(\varphi(x)=e^x-1\) perform best over all considered cases.
\end{itemize}
\FloatBarrier
\begin{figure}
    \centering
    \includegraphics[width = 0.78\textwidth]{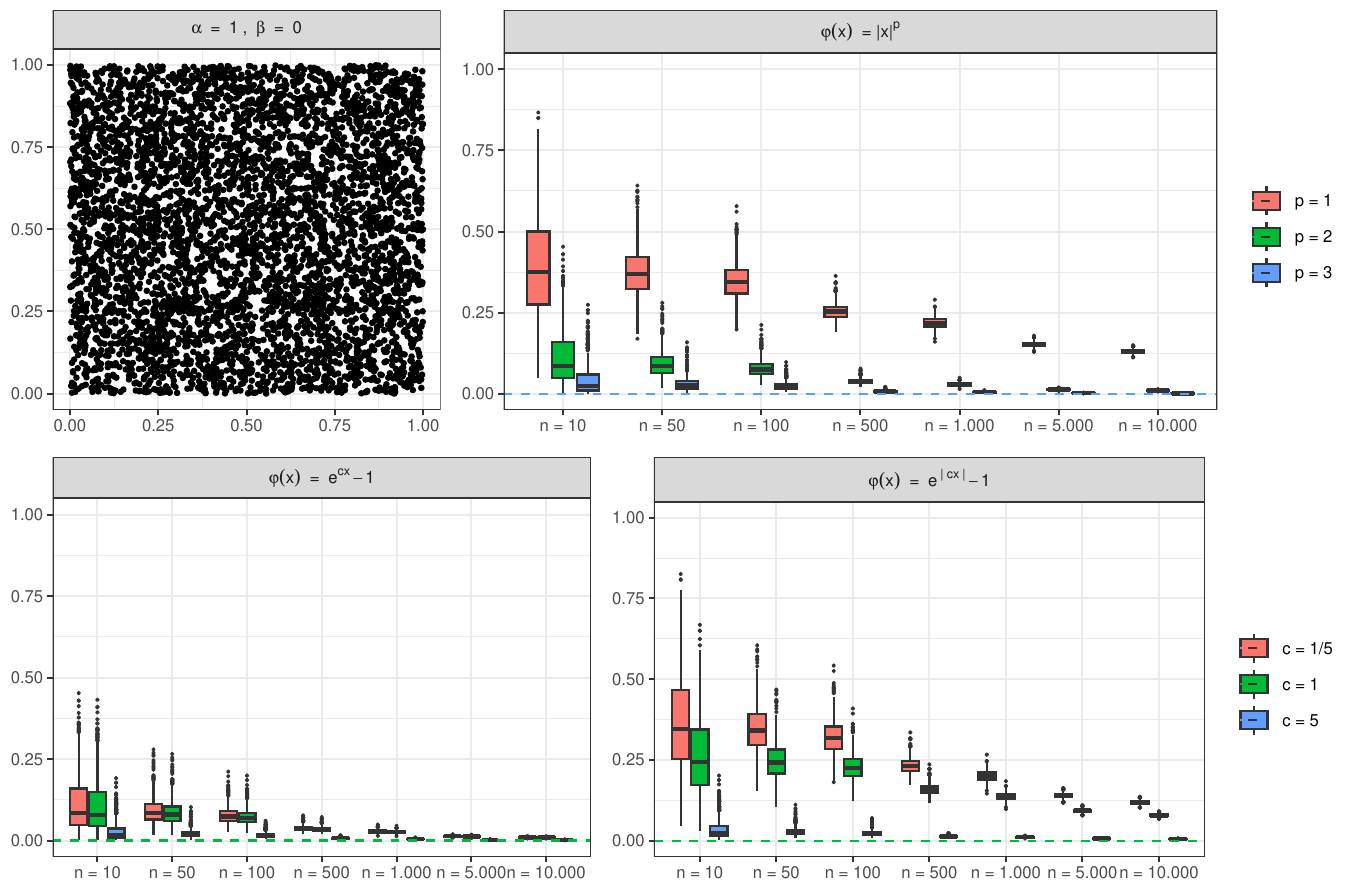}
    \caption{Simulation Results for the MO copula with parameters $\alpha = 1$ and $\beta = 0$, i.e., for the independence copula. The top left panel depicts a 
    random sample of size $n=5.000$ from this distribution. 
    The other panels contain boxplots of the estimations \(\Lambda_{\varphi}(\Ch_{N(n)}(E_n))\) according to Theorem \ref{theestphi} based on $1,000$ runs for each sample size \(n\) and for each convex function $\varphi$. The dashed lines 
    represent the true values of $\Lambda_\varphi(A)$.}
    \label{MO_1_0.pdf}
\end{figure}

\begin{figure}
    \centering
    \includegraphics[width = 0.78\textwidth]{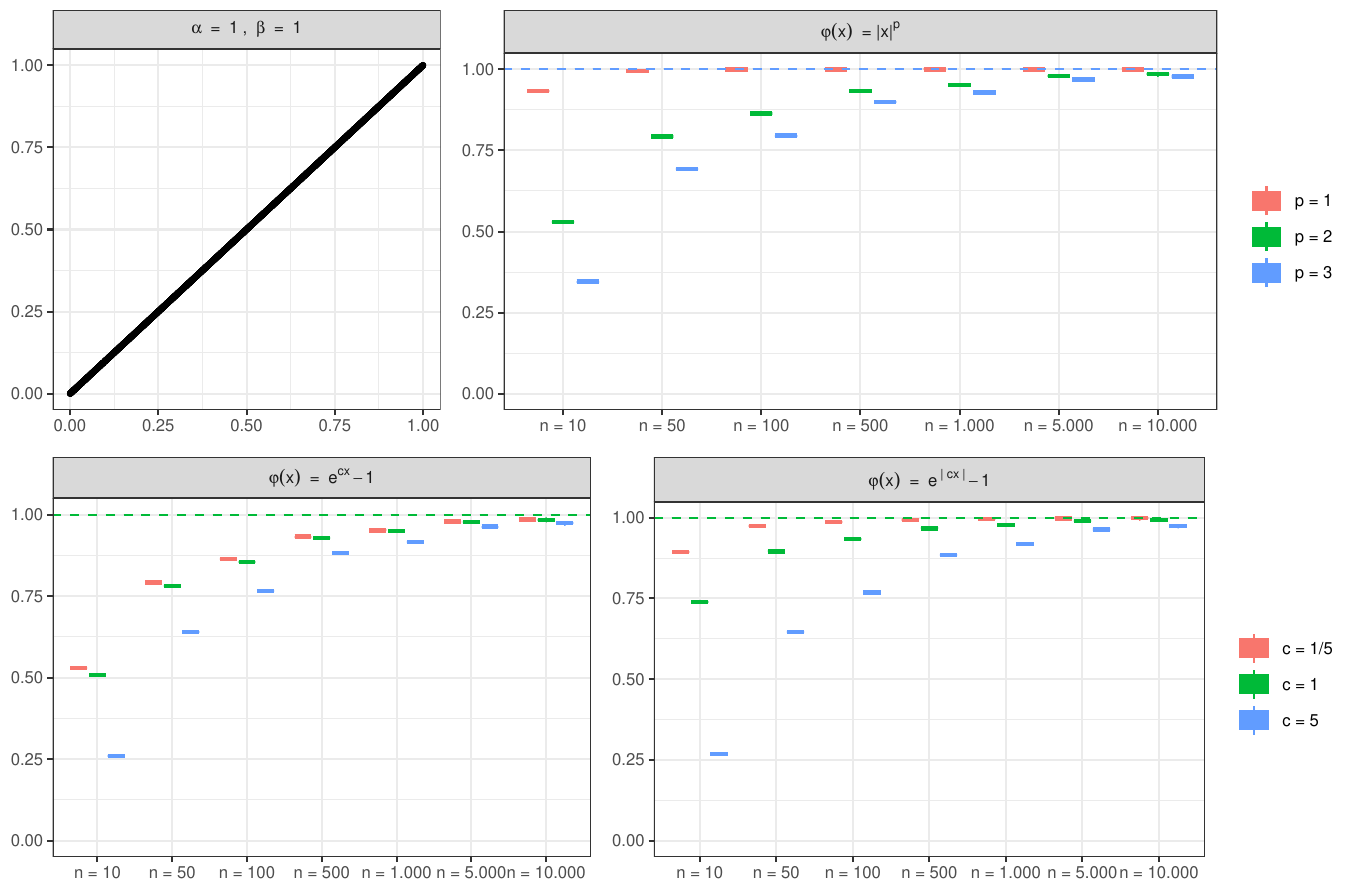}
    \caption{Simulation Results for the MO Copula with parameters $\alpha = 1$ and $\beta = 1$, i.e., for the case of perfect positive dependence.The top left panel depicts a 
    random sample of size $n=5.000$ from this distribution. 
    The other panels contain boxplots of the estimations \(\Lambda_{\varphi}(\Ch_{N(n)}(E_n))\) according to Theorem \ref{theestphi} based on $1,000$ runs for each sample size \(n\) and for each convex function $\varphi$. The dashed lines 
    represent the true values of $\Lambda_\varphi(A)$.}
    \label{MO_1_1.pdf}
\end{figure}

\begin{figure}
    \centering
    \includegraphics[width = 0.78\textwidth]{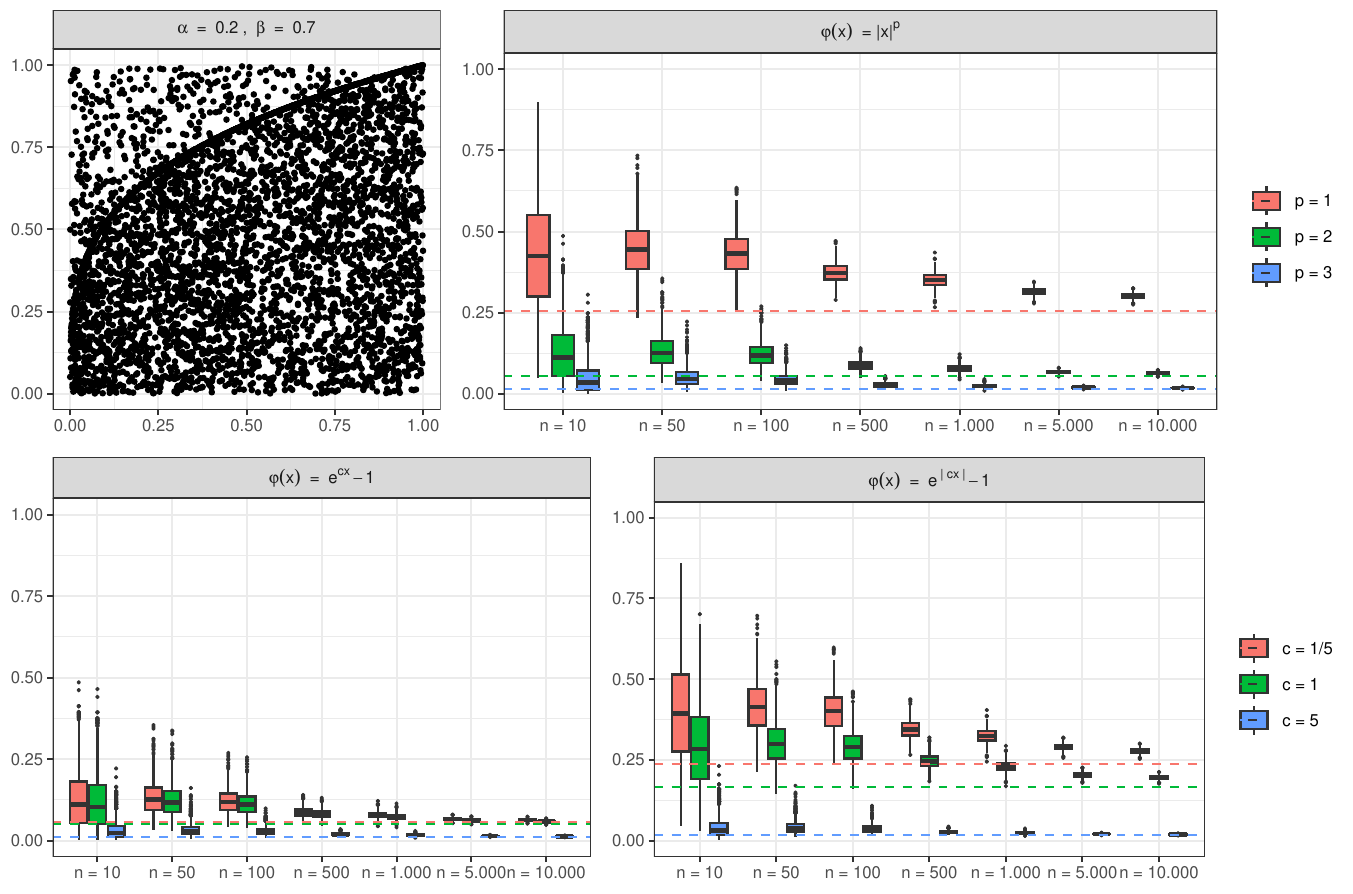}
    \caption{Simulation Results for the MO Copula with parameters $\alpha = 0.2$ and $\beta = 0.7$. The top left panel depicts a 
    random sample of size $n=5.000$ from this distribution. 
    The other panels contain boxplots of the estimations \(\Lambda_{\varphi}(\Ch_{N(n)}(E_n))\) according to Theorem \ref{theestphi} based on $1,000$ runs for each sample size \(n\) and for each convex function $\varphi$. The dashed lines 
    represent the true values of $\Lambda_\varphi(A)$.}
    \label{MO_0.2_0.7.pdf}
\end{figure}

\begin{figure}
    \centering
    \includegraphics[width = 0.78\textwidth]{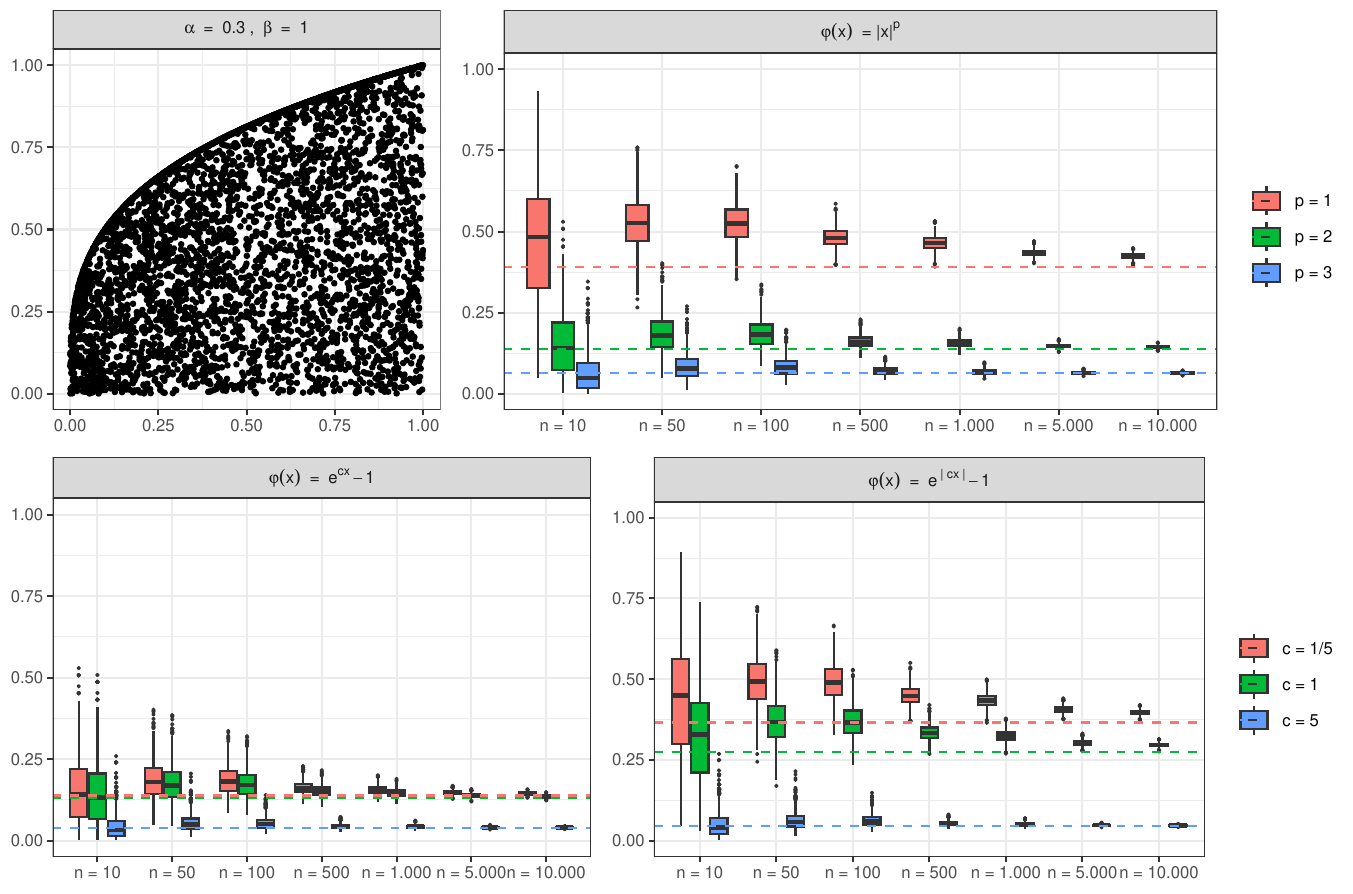}
    \caption{Simulation Results for the MO Copula with $\alpha = 0.3$ and $\beta = 1$.
    The top left panel depicts a 
    random sample of size $n=5.000$ from this distribution. 
    The other panels contain boxplots of the estimations \(\Lambda_{\varphi}(\Ch_{N(n)}(E_n))\) according to Theorem \ref{theestphi} based on $1,000$ runs for each sample size \(n\) and for each convex function $\varphi$. The dashed lines 
    represent the true values of $\Lambda_\varphi(A)$.}
    \label{MO_0.3_1.pdf}
\end{figure}
\FloatBarrier
\subsection{Real Data Example}\label{sec:ex_real}
As a real data example we consider a dataset containing gathered by~\cite{physionet_daten} and available online at PhysioNet.org, see \cite{physionet}. The dataset contains, among other variables, $101$ complete observations of pregnant women's age, BMI, blood pressure, blood glucose level and thickness of visceral adipose tissue as well as the birth weight of the child. We choose the birth weight of the child as the endogenous variable \(Y\) 
and each of the other variables as the exogenous one. Figure~\ref{Data_example} 
illustrates the empirical checkerboard copula $\Ch_{N(n)}(E_n))$ for 
BMI as exogenous variable.  
Table~\ref{Example_table} summarizes the obtained values of 
$\Lambda_{\varphi}(Y|X)$ for all exogenous variables. Interpretation of the different values is not straightforward since they depend on the shape of \(\varphi\) (unless they are exactly $0$ or $1$).
Due to the small sample size the calculated estimates might still differ 
from the underlying true value. Surprisingly, however, when sorting the variables by 
their estimated $\Lambda_{\varphi}(Y|X)$ value, the order is the same for all
convex functions $\varphi$, i.e., for all choices of $\varphi$ 
the variable importance (w.r.t. $Y$) is the same. 

\begin{figure}[h!]
    \centering
    \includegraphics[width = 0.78\textwidth]{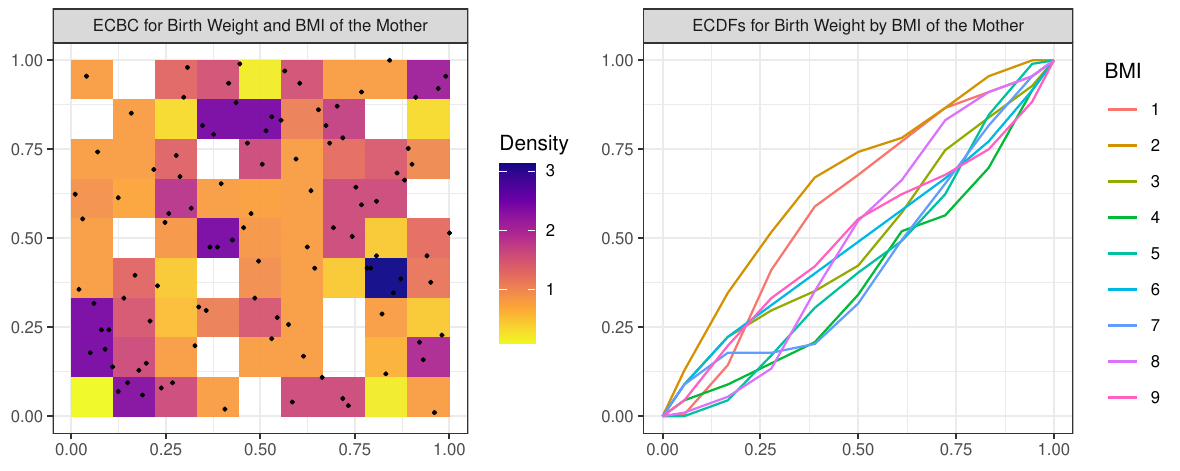}
    \caption{The left panel depicts the Empirical Checkerboard Copula $\Ch_{N(n)}(E_n))$ with resolution $N=9$ and together with the pseudo-observations (a.k.a. psuedo-observations) for Birth Weight ($y$-axis) and BMI of the mother ($x$-axis). The right plot contains the resulting conditional distribution functions $F_{Y|X=x}$ for the $9$ different BMI regions corresponding to the $9$ vertical stripes on the left plot. $\Lambda_{\Delta}(Y|X)$ is computed by taking the average pairwise distances $\Delta(F_i, F_j),\, i,j \in \{1,\ldots, 9\}$ over these $9$ ECDFs and subsequently dividing by $\alpha_{\Delta}$.}
    \label{Data_example}
\end{figure}

\begin{table}[]
    \centering
    \caption{Values of $\Lambda_{\varphi}(Y|X)$ for various choices of the exogenous variable \(X\) and for different convex functions $\varphi$. 
    In each of the cases the endogenous variable \(Y\) is the birth weight.}
   \begin{tabular}{lrrr}
        \hline
        & \multicolumn{3}{c}{$\varphi(x)$}\\
        \cline{2-4}
        Exogenous Variable & $|x|$ & $e^{|x|} - 1$ & $e^{x} - 1$\\
        \hline
        Diastolic BP &$0.354$&$0.231$&$0.072$\\
        BMI &$0.326$&$0.210$&$0.060$\\
        Visceral Adipose Tissue &$0.304$&$0.195$&$0.052$\\
        Systolic BP &$0.291$&$0.186$&$0.049$\\
        Blood Glucose &$0.214$&$0.134$&$0.027$\\
        Age &$0.137$&$0.083$&$0.012$\\
        \hline
    \end{tabular}
    \label{Example_table}
\end{table}
\FloatBarrier
\noindent
\textbf{Funding}\\
The authors gratefully acknowledge the support of the WISS 2025 project `IDA-lab Salzburg' (20204-WISS/225/197-2019 and 20102-F1901166-KZP).
The first and the third author further acknowledge the support of the Austrian Science Fund (FWF) project
{P 36155-N} \emph{ReDim: Quantifying Dependence via Dimension Reduction}.

\bibliographystyle{plain} % Style BST file (imsart-number.bst or imsart-nameyear.bst)
%\bibliography{Ansari_bib}       % Bibliography file (usually '*.bib')

\begin{thebibliography}{10}

\bibitem{Ansari-2019}
Jonathan Ansari.
\newblock {\em Ordering risk bounds in partially specified factor models}.
\newblock Freiburg im Breisgau: Univ. Freiburg, Fakult{\"a}t f{\"u}r Mathematik
  und Physik (Diss.), 2019.

\bibitem{Ansari-Fuchs-2022}
Jonathan Ansari and Sebastian Fuchs.
\newblock A simple extension of azadkia \& chatterjee's rank correlation to a
  vector of endogenous variables.
\newblock {\em Available at \url{} arXiv:2212.01621}, 2022.

\bibitem{Ansari-2021}
Jonathan Ansari and Ludger R{\"u}schendorf.
\newblock Sklar's theorem, copula products, and ordering results in factor
  models.
\newblock {\em Depend. Model.}, 9:267--306, 2021.

\bibitem{chatterjee2021}
M.~Azadkia and S.~Chatterjee.
\newblock A simple measure of conditional dependence.
\newblock {\em Ann. Stat.}, 49(6):3070--3102, 2021.

\bibitem{bickel2022}
P.J. Bickel.
\newblock Measures of independence and functional dependence.
\newblock {\em Available at \url{https://arxiv.org/abs/2206.13663v1}}, 2022.

\bibitem{chatterjee2020}
S.~Chatterjee.
\newblock A new coefficient of correlation.
\newblock {\em J. Amer. Statist. Ass.}, 116(536):2009--2022, 2020.

\bibitem{cover2006}
T.~M. Cover and J.~A. Thomas.
\newblock {\em Elements of Information Theory}.
\newblock John Wiley \& Sons, Hoboken, 2006.

\bibitem{deb2020}
N.~Deb, P.~Ghosal, and B.~Sen.
\newblock Measuring association on topological spaces using kernels and
  geometric graphs.
\newblock {\em Available at \url{http://128.84.4.18/abs/2010.01768}}, 2020.

\bibitem{siburg2013}
H.~Dette, K.~F. Siburg, and P.~A. Stoimenov.
\newblock A copula-based non-parametric measure of regression dependence.
\newblock {\em Scand. J. Statist.}, 40(1):21--41, 2013.

\bibitem{ding2017}
A.A. Ding, J.G. Dy, Y~Li, and Y.~Chang.
\newblock A robust-equitable measure for feature ranking and selection.
\newblock {\em J. Mach. Learn. Res.}, 18:1--46, 2017.

\bibitem{Dwork-2006}
Cynthia Dwork, Frank McSherry, Kobbi Nissim, and Adam Smith.
\newblock Calibrating noise to sensitivity in private data analysis.
\newblock In {\em Theory of cryptography. Third theory of cryptography
  conference, TCC 2006, New York, NY, USA, March 4--7, 2006. Proceedings.},
  pages 265--284. Berlin: Springer, 2006.

\bibitem{Evfimievski-2003}
Alexandre Evfimievski, Johannes Gehrke, and Ramakrishnan Srikant.
\newblock Limiting privacy breaches in privacy preserving data mining.
\newblock In {\em Proceedings of the twenty-second ACM SIGMOD-SIGACT-SIGART
  symposium on Principles of database systems}, pages 211--222, 2003.

\bibitem{sfx2022phi}
S.~Fuchs.
\newblock Quantifying directed dependence via dimension reduction.
\newblock {\em J. Multivariate Anal., to appear}, 2023.

\bibitem{Furman-2017}
Edward Furman, Ruodu Wang, and Ricardas Zitikis.
\newblock Gini-type measures of risk and variability: Gini shortfall, capital
  allocations, and heavy-tailed risks.
\newblock {\em Journal of Banking \& Finance}, 83:70--84, 2017.

\bibitem{gamboa2020}
F.~Gamboa, P.~Gremaud, T.~Klein, and A.~Lagnoux.
\newblock Global sensitivity analysis: A novel generation of mighty estimators
  based on rank statistics.
\newblock {\em Bernoulli}, 28(4):2345--2374, 2022.

\bibitem{Genest-2017}
Christian Genest, Johanna~G. Neslehov{\'a}, and Bruno R{\'e}millard.
\newblock Asymptotic behavior of the empirical multilinear copula process under
  broad conditions.
\newblock {\em J. Multivariate Anal.}, 159:82--110, 2017.

\bibitem{physionet}
A.~Goldberger, L.~Amaral, L.~Glass, J.~Hausdorff, P.~C. Ivanov, R.~Mark, J.~E.
  Mietus, B.~Moody, G., and K.~\& Stanley H.~E. Peng, C.
\newblock Physiobank, physiotoolkit, and physionet: Components of a new
  research resource for complex physiologic signals.
\newblock {\em Circulation [Online]}, 101(23):e215--e220, 2000.
\newblock https://doi.org/10.13026/p729-7p53.

\bibitem{fgwt2021}
F.~Griessenberger, R.R. Junker, and W.~Trutschnig.
\newblock On a multivariate copula-based dependence measure and its estimation.
\newblock {\em Electron. J. Statist.}, 16:2206--2251, 2022.

\bibitem{fan2022A}
F.~Han and Z.~Huang.
\newblock {A}zadkia-{C}hatterjee’s correlation coefficient adapts to manifold
  data.
\newblock {\em Available at https://arxiv.org/abs/2209.11156v1}, 2022.

\bibitem{deb2020b}
Z.~Huang, N.~Deb, and B.~Sen.
\newblock Kernel partial correlation coefficient — a measure of conditional
  dependence.
\newblock {\em J. Mach. Learn. Res.}, 23(216):1--58, 2022.

\bibitem{JSV}
Paul Janssen, Jan Swanepoel, and No{\"e}l Veraverbeke.
\newblock Large sample behavior of the {Bernstein} copula estimator.
\newblock {\em J. Stat. Plann. Inference}, 142(5):1189--1197, 2012.

\bibitem{JGT}
R.R. Junker, F.~Griessenberger, and W.~Trutschnig.
\newblock Estimating scale-invariant directed dependence of bivariate
  distributions.
\newblock {\em Comput. Statist. Data Anal.}, 153:Article ID 107058, 22 pages,
  2020.

\bibitem{KFT}
T.~Kasper, S.~Fuchs, and W.~Trutschnig.
\newblock On weak conditional convergence of bivariate {A}rchimedean and
  extreme value copulas, and consequences to nonparametric estimation.
\newblock {\em Bernoulli}, 27:2217--2240, 2021.

\bibitem{kinney2014}
J.B. Kinney and G.S Atwal.
\newblock Equitability, mutual information, and the maximal information
  coefficient.
\newblock {\em Proc. Natl. Acad. Sci. USA}, 111:3354--3359, 2014.

\bibitem{LMT}
X.~Li, P.~Mikusi{\'n}ski, and M.~D. Taylor.
\newblock Strong approximation of copulas.
\newblock {\em J. Math. Anal. Appl.}, 225(2):608--623, 1998.

\bibitem{MFFT}
Thomas Mroz, Juan Fern{\'a}ndez-S{\'a}nchez, Sebastian Fuchs, and Wolfgang
  Trutschnig.
\newblock On distributions with fixed marginals maximizing the joint or the
  prior default probability, estimation, and related results.
\newblock {\em J. Stat. Plann. Inference}, 223:33--52, 2023.

\bibitem{Mueller-Scarsini-2005}
Alfred M{\"u}ller and Marco Scarsini.
\newblock Archimedean copulae and positive dependence.
\newblock {\em J. Multivariate Anal.}, 93(2):434--445, 2005.

\bibitem{Nelsen-2006}
R.~B. Nelsen.
\newblock {\em An Introduction to Copulas.}
\newblock Springer, New York, 2nd edition, 2006.

\bibitem{Munk-2023}
T.G. Nies, T.~Staudt, and A.~Munk.
\newblock Transport dependency: Optimal transport based dependency measures.
\newblock {\em {Available at \url{https://arxiv.org/abs/2105.02073}}, Volume =
  {}, Pages = {}, Year = {2023}}.

\bibitem{physionet_daten}
A.~d.~S Rocha, L.~von Diemen, D.~Kretzer, S.~Matos, and J.~A.
  Rombaldi~Bernardi, J. \&~Magalhaes.
\newblock Visceral adipose tissue measurements during pregnancy (version
  1.0.0), 2020.
\newblock https://doi.org/10.13026/p729-7p53.

\bibitem{Rohde-2020}
Angelika Rohde and Lukas Steinberger.
\newblock Geometrizing rates of convergence under local differential privacy
  constraints.
\newblock {\em Ann. Stat.}, 48(5):2646--2670, 2020.

\bibitem{Ru-2013}
Ludger R{\"u}schendorf.
\newblock {\em Mathematical risk analysis. {Dependence}, risk bounds, optimal
  allocations and portfolios}.
\newblock Springer Ser. Oper. Res. Financ. Eng. Berlin: Springer, 2013.

\bibitem{Shaked-2013}
Moshe Shaked, Miguel~A. Sordo, and Alfonso Su{\'a}rez-Llorens.
\newblock A global dependence stochastic order based on the presence of noise.
\newblock In {\em Stochastic Orders in Reliability and Risk. In Honor of
  Professor Moshe Shaked.}, pages 3--39. New York, NY: Springer, 2013.

\bibitem{emura2021}
J.-H. Shih and T.~Emura.
\newblock On the copula correlation ratio and its generalization.
\newblock {\em J. Multivariate Anal.}, 182:Article ID 104708, 2021.

\bibitem{strothmann2022}
C.~Strothmann, H.~Dette, and K.F. Siburg.
\newblock Rearranged dependence measures.
\newblock {\em Bernoulli}, pages to appear, Available at
  \url{https://arxiv.org/abs/2201.03329v1}, 2022.

\bibitem{sungur2005}
E.~A. Sungur.
\newblock A note on directional dependence in regression setting.
\newblock {\em Comm. Statist. Theory Methods}, 34:1957--1965, 2005.

\bibitem{wt2011}
W.~Trutschnig.
\newblock On a strong metric on the space of copulas and its induced dependence
  measure.
\newblock {\em J. Math. Anal. Appl.}, 384(2):690--705, 2011.

\bibitem{wiesel2022}
J.C.W. Wiesel.
\newblock Measuring association with {W}asserstein distances.
\newblock {\em Bernoulli}, 28:2816--2832, 2022.

\end{thebibliography}

\end{document}